\documentclass[12pt]{amsart}
\usepackage{amssymb}

\numberwithin{equation}{section}

\theoremstyle{plain}
\newtheorem{theorem}{Theorem}[section]

\newtheorem{lemma}[theorem]{Lemma}
\newtheorem{corollary}[theorem]{Corollary}
\newtheorem{proposition}[theorem]{Proposition}
\newtheorem{theorem-definition}[theorem]{Theorem and Definition}
\newtheorem{lemma-definition}[theorem]{Lemma and Definition}
\newtheorem*{theorem*}{theorem}
\newtheorem*{theoremA*}{Theorem A}
\newtheorem*{corollaryB*}{Corollary B}
\newtheorem*{theoremC*}{Theorem C}
\newtheorem*{satellite*}{Satellite to the Main Theorem}

\newtheorem{realcrit}[theorem]{Realizability Criterion}

\theoremstyle{definition}
\newtheorem{definition}[theorem]{Definition}
\newtheorem{definitions}[theorem]{Definitions}
\newtheorem{observation}[theorem]{Observation}
\newtheorem*{observation*}{Observation}

\newtheorem{example}[theorem]{Example}
\newtheorem{examples}[theorem]{Examples}

\newtheorem{remarks}[theorem]{Remarks}
\newtheorem*{remark*}{Remark}

\newtheorem{subsec}[theorem]{}

\input xy \xyoption{matrix} \xyoption{arrow}
          \xyoption{curve}  \xyoption{frame}
\def\edge{\ar@{-}}
\def\dttdar{\ar@{.>}}
\def\drbl{\save+<0ex,-2ex> \drop{\bullet} \restore}

\def\horizpool#1{  \save[0,0]+(0,3);[0,0]+(0,3) **\crv{~*=<2.5pt>{.} [0,#1]+(0,3) &[0,#1]+(3,3) &[0,#1]+(3,0) &[0,#1]+(3,-3) &[0,#1]+(0,-3) &[0,0]+(0,-3) &[0,0]+(-3,-3) &[0,0]+(-3,0) &[0,0]+(-3,3)}\restore }
\def\dashedge{\ar@{--}}

\def\dshdar{\ar@{-->}}

\def\dropup#1{\save+<0ex,4ex> \drop{#1} \restore}

\def\fatedge#1{\edge@<0.025pc>[#1] \edge@<-0.025pc>[#1] \edge@<0.05pc>[#1]  \edge@<-0.05pc>[#1] \edge[#1]}

\def\vsubseteq{\hbox{$\bigcup$\kern0.1em\raise0.05ex\hbox{$\textstyle|$}}}
\def\smvsubseteq{\hbox{$\ssize\bigcup$\kern0.03em\raise0.05ex\hbox{$\ssize|$}}}

\def\la{{\Lambda}}
\DeclareMathOperator \trunc {trunc}
\def\latrunc{\la_{\trunc}}
\def\Ptrunc{P_{\trunc}}
\def\lamod{\Lambda\mbox{\rm-mod}}
\def\latruncmod{\latrunc\mbox{\rm-mod}}

\def\SS{{\mathbb S}}
\def\ZZ{{\mathbb Z}}
\def\NN{{\mathbb N}}

\DeclareMathOperator \len {length}

\DeclareMathOperator \Aut {Aut}

\DeclareMathOperator \soc {soc}

\DeclareMathOperator \strt {start}

\DeclareMathOperator \term {end}

\DeclareMathOperator \GL {GL}

\DeclareMathOperator \Img {Im}
\DeclareMathOperator \Seq {\mathbf {Seq}}

\DeclareMathOperator \End {End}

\DeclareMathOperator \Filt {\mathbf{Filt}}

\DeclareMathOperator \udim {\underline{dim}}

\DeclareMathOperator \trdeg {tr deg}

\def\C{{\mathcal C}}
\def\D{{\mathcal D}}

\def\G{{\mathcal G}}

\def\K{{\mathcal K}}

\def\P{{\mathcal P}}

\def\R{{\mathcal R}}
\def\S{{\sigma}}

\def\U{{\mathfrak U}}

\def\bB{\mathbf {B}}

\def\bP{\mathbf {P}}

\def\bd{\mathbf {d}}

\def\bp{\mathbf {p}}
\def\bq{\mathbf {q}}

\def\SStilde{\widetilde{\SS}}

\def\bP{\mathbf{P}}
\def\SShat{\widehat{\SS}}
\def\SStilde{\widetilde{\SS}}

\def\bz{\mathbf{z}}

\DeclareMathOperator \Rep {\mathbf{Rep}}
\DeclareMathOperator \laySS {\Rep \SS}
\newcommand \modlad {\Rep_\bd(\Lambda)}

\DeclareMathOperator \scaledrep {\Delta-{\Rep}}

\DeclareMathOperator \Corep {\mathbf{Corep}}
\DeclareMathOperator \Cofilt {\mathbf{Cofilt}}

\DeclareMathOperator \Hom {Hom}

\DeclareMathOperator \Gr {Gr}

 \DeclareMathOperator \biggrass {GRASS}

\DeclareMathOperator \Schu {\bf {Schu}}

\DeclareMathOperator \flag {\mathfrak{Flag}}

\DeclareMathOperator \GRASS {GRASS}

\def\grassbd{\GRASS_\bd(\Lambda)}

\def\biggrassSS{\GRASS(\SS)}
\def\biggrassS{\GRASS(\S)}

\def\Gabull{\Gamma_\bullet}

\title{Closures in varieties of representations and irreducible components}
\author[K. R. Goodearl]{K. R. Goodearl}
\address{
Department of Mathematics \\
University of California\\
Santa Barbara, CA 93106 \\
U.S.A.
}
\email{goodearl@math.ucsb.edu}
\author[B. Huisgen-Zimmermann]{B. Huisgen-Zimmermann}
\address{
Department of Mathematics \\
University of California\\
Santa Barbara, CA 93106 \\
U.S.A.
}
\email{birge@math.ucsb.edu}

\begin{document}
%%%%%%%%%%%%%%%%%%%%%%%%%%%%%%%%%%%%%%%%%%%%%%%%%%%

\dedicatory{Dedicated to the memory of Peter Gabriel}
  
\begin{abstract} For any truncated path algebra $\la$ of a quiver, we classify, by way of representation-theoretic invariants, the irreducible components of the parametrizing varieties $\modlad$ of the $\la$-modules with fixed dimension vector $\bd$.  In this situation, the components of $\modlad$ are always among the closures $\overline{\laySS}$, where $\SS$ traces the semisimple sequences with dimension vector $\bd$, and hence the key to the classification problem lies in a characterization of these closures.    

Our first result concerning closures actually addresses arbitrary basic finite dimensional algebras over an algebraically closed field.  In the general case, it corners the closures $\overline{\laySS}$ by means of module filtrations ``governed by $\SS$"; in case $\la$ is truncated, it pins down the $\overline{\laySS}$ completely.

The analysis of the varieties $\overline{\laySS}$ leads to a novel upper semicontinuous module invariant which provides an effective tool towards the detection of components of $\modlad$ in general.  It detects all components when $\la$ is truncated.
\end{abstract}



\maketitle

%%%%%%%%%%%%%%%%%%%%%%%%%%%%%
\section{Introduction}
\label{intro}

By strong consensus, a classification of all indecomposable finite dimensional representations of a finite dimensional algebra $\la$ is  an unattainable goal in general.  A far more promising alternative to this impossibly comprehensive problem is that of {\it generically\/} classifying the finite dimensional $\la$-modules.  This amounts to understanding the generic structure of the modules in the irreducible components of the varieties $\modlad$ which parametrize the $\la$-modules with dimension vector $\bd$. By its very nature, this quest comes paired with the task of pinning down the irreducible components of the $\modlad$ in representation-theoretic terms.

In the present article, the component problem is solved for arbitrary truncated path algebras $\la$ over an algebraically closed field $K$.   In tandem, significant headway is made towards determining the generic features of the modules in the components.  

The classification of the components, in turn, relies on a characterization of the modules in the closures of certain representation-theoretically defined locally closed subvarieties of $\modlad$.  Our initial round of results regarding such closures, including the description of an associated upper semicontinuous module invariant which serves to test for inclusions, holds for arbitrary basic finite dimensional $K$-algebras. The findings lead to partial lists of components in this broad scenario. The results become tight on specialization to the truncated case.

Throughout, we assume $K$ to be an algebraically closed field and $\la$ a basic finite dimensional $K$-algebra.  This means that, up to isomorphism, $\la = KQ/I$ for a quiver $Q$ and an admissible ideal $I$ in the path algebra.  The maximal length of a path in $KQ \setminus I$ will be denoted by $L$; in other words, $L$ is minimal with respect to $J^{L+1} = 0$, where $J$ is the Jacobson radical of $\la$.  Consequently, the {\it radical layering\/} $\SS(M)$ of a $\la$-module $M$ has no more than $L+1$ nonzero entries:  $\SS(M) = (J^l M / J^{l+1} M)_{0 \le l \le L}$.
By $\modlad$, we  denote the standard affine variety parametrizing the $\la$-modules with dimension vector $\bd$. This variety is partitioned into finitely many locally closed subvarieties $\laySS$ corresponding to the {\it semisimple sequences $\SS$ with dimension vector $\bd$\/}; these are the sequences $\SS = (\SS_0, \dots, \SS_L)$ of (isomorphism classes of) semisimple $\la$-modules with $\udim \SS  := \sum_{0 \le l \le L} \udim \SS_l = \bd$; here $\laySS$ consists of those points $x$ in $\modlad$ which represent modules $M_x$ with $\SS(M_x) = \SS$.  

The closures $\overline{\laySS}$ are relevant towards the problem of describing the irreducible components of $\modlad$:  Indeed, it is readily seen that the components of the ambient variety are always among those of the $\overline{\laySS}$, where $\SS$ traces the $\bd$-dimensional semisimple sequences.  Less obviously, the components of the subvarieties $\laySS$, and hence those of their closures, may be obtained from quiver and relations by way of a straightforward algorithm, each component tagged by a ``generic minimal projective presentation" of the modules it encodes (see \cite{BHT} and \cite{hier}).  Identifying the components of $\modlad$ thus amounts to a sorting problem:  For which components $\C$ of $\laySS$ is the closure $\overline{\C}$ {\it maximal\/} among the irreducible subsets of $\modlad$?   This is an extremely taxing question in general, calling for a thorough understanding of the boundaries of the varieties $\laySS$. 

Our strategy consists of moving back and forth between the varieties $\modlad$ and $\grassbd$; the latter is a closed subvariety of a vector space Grassmannian which parametrizes the modules with dimension vector $\bd$ by suitable submodules of a projective cover of the semisimple module with this dimension vector (see Section \ref{sec2} and \cite{hier, GoHZ}). The irrreducible components of the projective variety $\grassbd$ may be studied by ``spreading them out" within a suitable flag variety (Theorem \ref{prop3.9}), and the subsequent transfer of information $\grassbd \longleftrightarrow \modlad$ is modeled on Gabriel's influential work in \cite{Gab}. In a first step, we show:

\begin{theoremA*}  {\rm (cf.~\ref{thm3.8} and \ref{thm4.3}; see also \ref{rem3.7}(4).)}  Let $\la = KQ/I$ be a path algebra modulo relations, $L+1$ its Loewy length, and $\SS = (\SS_0, \dots, \SS_L)$ a $\bd$-dimensional semisimple sequence in $\lamod$.   Then every module in the closure $\overline{\laySS}$ has a filtration by submodules, 
$$M = M_0 \supseteq M_1 \supseteq \cdots \supseteq M_{L+1} = 0,$$
 which is ``governed by $\SS$" in the sense that  the quotients $M_l / M_{l+1}$ are semisimple and isomorphic to $\SS_l$, respectively.  
 In fact, the set $\Filt \SS$ consisting of those points in $\modlad$ that correspond to modules with at least one filtration governed by $\SS$ is always closed. 
 \smallskip

If $\la$ is a truncated path algebra, i.e., $\la = KQ/ \langle \text{all paths of length}\ L+1 \rangle$, and $\laySS$ is nonempty, then $$\overline{\laySS} = \Filt \SS.$$      
\end{theoremA*} 

For general $\la$, the inclusion $\overline{\laySS} \subseteq \Filt \SS$ may be proper.  The question of whether a point in $\modlad$ belongs to $\Filt \SS$ may be answered by testing for similarity of certain matrices.  By contrast, to date, there is no algorithm for deciding whether a module belongs to $\overline{\laySS}$.   

A semisimple sequence $\SS$ is called {\it realizable\/} if $\laySS \ne \varnothing$.  (In case $\la$ is a truncated path algebra, realizability is checked via mere inspection of the quiver; see \cite[Criterion 3.2]{irredcompI} and \ref{crit4.1} below.) 

\begin{corollaryB*}  {\rm (cf.~\ref{cor3.11}.)} For $M \in \lamod$, let $\, \Gamma(M)$ be  the number of those realizable semisimple sequences that govern at least one filtration of $M$.  Then 
$$\Gabull: \modlad \rightarrow \NN, \quad x \mapsto \Gamma(M_x),$$
is an upper semicontinuous function.   

In particular:  Whenever $\C$ is an irreducible component of some $\laySS$ such that $1 \in \Gabull(\C)$, the closure $\overline{\C}$ is an irreducible component of $\modlad$.
\end{corollaryB*}

In the second part of the paper, we derive consequences for truncated path algebras.  As is suggested by Theorem A, the component problem simplifies considerably in this situation.  Notably, the subvarieties $\overline{\laySS}$ are all irreducible, and generic minimal projective presentations of the modules in $\laySS$ are immediate from quiver and Loewy length (see \cite[Section 5]{BHT} and Section 5.A  below). In some prominent special cases, particularly manageable solutions to the problem of sifting out the inclusion-maximal ones among the closures $\overline{\laySS}$ are already available (see \cite{irredcompI, irredcompII}):  For instance, if $\la$ is either local or based on an acyclic quiver $Q$, the semisimple sequences singled out by the minimal values of the following upper semicontinuous map furnish a complete, nonrepetitive parametrization of the components $\overline{\laySS}$ of $\modlad$: 
\begin{equation}  \label{Thetadef}
\Theta = (\SS_\bullet, \SS^*_\bullet) : \modlad \rightarrow \Seq(\bd) \times \Seq(\bd), \ \ \ x \mapsto \bigl(\SS(M_x), \SS^*(M_x) \bigr);
\end{equation}
here the codomain of $\Theta$ is partially ordered by the componentwise dominance order on the set $\Seq(\bd)$ of all $\bd$-dimensional semisimple sequences (see Section \ref{sec2}), and $\SS^*(M_x)$ stands for the socle layering of the module $M_x$ (the dual of the radical layering).  The unique minimal sequence $\SS^*(M_x)$ attained on $\laySS$, that is, the generic socle layering of the modules in $\overline{\laySS}$, is supplied by a closed formula based on $\SS$, $Q$ and $L$ \cite[Theorem 3.8]{irredcompII}, which makes the $\Theta$-test very user-friendly.
But for general truncated $\la$, the map $\Theta$ fails to detect all components, even when supplemented by further standard semicontinuous module invariants, such as path ranks or assortments of annihilator dimensions.  The map $\Gabull$, on the other hand, compensates for the blind spots of $\Theta$:  

\begin{theoremC*} {\rm (cf.~\ref{thm4.5}.)} If $\la$ is any truncated path algebra, the irreducible components of $\modlad$ are precisely those closures $\overline{\laySS}$ on which $\Gabull$ attains the value $1$. 

 In other words, $\overline{\laySS}$ is maximal among the irreducible subsets of $\modlad$ if and only if there exists a module $N$ in $\laySS$ such that $N \supseteq JN \supseteq \cdots \supseteq J^{L+1} N$ is the only filtration of $N$ which is governed by a realizable semisimple sequence.
\end{theoremC*}

In deciding which semisimple sequences $\SS$ are the generic radical layerings of the irreducible components of $\modlad$, Theorem C thus permits exclusive reliance on $\Gabull$.  However, in practice, combining $\Gabull$ with the test map $\Theta$ is considerably more efficient.
\medskip

In the pursuit of a generic approach to the structure of $\la$-modules, the hereditary case, pioneered in \cite{KacI, KacII} and \cite{Scho}, serves as a model.
We further point to a selection of existing contributions to the component problem over non-hereditary algebras: General tools were developed in \cite{CBS} and \cite{BHT}.  Solutions to the problem over specific classes of tame algebras were given in \cite{BaSch, CW, DoFl, GeiSchI, GeiSchII, Mor, RiRuSm, Schro} for instance; solutions for certain classes of wild non-hereditary algebras can be found in \cite{BCH, irredcompI, irredcompII}.  As is to be expected, meaningful classifications of the irreducible components of $\modlad$ in the quoted instances are throughout obtained via partial lists of generic properties of the modules in the components.  For a more detailed discussion of prior work on the topic we refer to the introduction of \cite{irredcompI}.

We add a few comments on the foundational nature of truncated path algebras with respect to the component problem.  Clearly, given an arbitrary basic $K$-algebra $\la = KQ/I$, there is a unique truncated path algebra $\latrunc$ having the same quiver and Loewy length as $\la$.  In the general situation, the varieties $\overline{\laySS}$ typically break up into multiple components.  Given that all of them are contained in irreducible components of $\Rep_\bd(\latrunc)$, it is advantageous to first determine the latter, say 
$$\overline{\Rep_{\latrunc}\SS^{(1)}} = \Filt_{\latrunc}(\SS^{(1)}),\ \ \dots\ ,\  \overline{\Rep_{\latrunc}\SS^{(m)}} = \Filt_{\latrunc}(\SS^{(m)}),$$ 
before aiming at the irreducible components of $\modlad$.  Indeed, this confines the need for size comparisons among the closures of components of the varieties $\Rep_\la \SS$  to the subvarieties $\Filt_{\latrunc} \SS^{(j)}\, \cap\, \modlad$; see  Section 6.B.  
%\medskip  

\subsection*{Overview} In Section \ref{sec2}, we provide background for the proofs of the main results and introduce a recurring example.  Section \ref{sec3} addresses the general case, where $\la$ is basic but otherwise unrestricted.  In Sections \ref{sec4} and \ref{sec5}, we apply the findings to truncated path algebras.  Section \ref{sec4} contains the announced classification of the irreducible components of $\modlad$, while in Section \ref{sec5}, we discuss generic modules and apply the results of Section \ref{sec4} towards interconnections among the components.  Section \ref{sec6}, finally, illustrates the theory and addresses the interplay $\modlad \longleftrightarrow \Rep_\bd(\latrunc)$.

%%%%%%%%%%%%%%%%%%%%%%%%%%%%%
\section{Conventions and prerequisites}
\label{sec2}

To repeat: Throughout, we assume $\la = KQ/I$ to be a basic finite dimensional algebra over $K = \overline{K}$ with Jacobson radical $J$ and Loewy length $L+1$.  The composition $pq$ of paths stands for ``$p$ after $q$" in case $\strt(p) = \term(q)$, while $pq=0$ in $KQ$ otherwise.  By $\latrunc$ we denote the truncated path algebra associated to $\la$, namely,  
 $$\latrunc = KQ/ \langle \text{the paths of length}\ L+1 \rangle;$$
we make no notational distinction between the $\la$- and $\latrunc$-structures of the objects in $\lamod$.   The vertices $e_1,\dots,e_n$ of $Q$ will be identified with the paths of length zero in $KQ$, as well as with the corresponding primitive idempotents in $\la$.  An element $x$ of a $\la$-module $M$ is said to be {\it normed by $e_i$\/} if $x= e_i x$, and a normed element in $M \setminus JM$ is called a {\it top element\/} of $M$. A {\it full sequence of top elements\/} of $M$ is a generating set of $M$ consisting of top elements which are $K$-linearly independent modulo $JM$. The simple module $\la e_i/ Je_i$ corresponding to the vertex  $e_i$ will be denoted by $S_i$, and isomorphic semisimple modules will be identified.

The {\it dominance order\/} on the set $\Seq(\bd)$ of all semisimple sequences with dimension vector $\bd$ is defined as follows:
$$(\SS_0, \dots, \SS_L) \le (\SS'_0, \dots, \SS'_L) \ \ \iff \ \ \bigoplus_{0\le j\le l} \SS_j \subseteq \bigoplus_{0\le j\le l} \SS'_j  \ \ \text{for} \ \ 0 \le l\le L.$$
Recall that the radical and socle layerings of a $\la$-module $M$ are denoted by $\SS(M)$ and $\SS^*(M)$. For basic properties of these semisimple sequences, we refer to \cite[Section 2.B]{irredcompI}. 

We fix our notation for the parametrizing varieties of the $\bd$-dimensional $\la$-modules. The affine variety $\modlad$ is 
$$\bigl\{ (x_\alpha)_{\alpha\in Q_1} \ \in \ \prod_{\alpha\in Q_1} \Hom_K \bigl( K^{d_{\strt(\alpha)}}, K^{d_{\term(\alpha)}} \bigr) \bigm| \text{the} \; x_\alpha \; \text{satisfy all relations in} \; I \bigr\},$$
where $Q_1$ is the set of arrows of $Q$. The orbits of the obvious conjugation action on $\modlad$ by the group $\GL(\bd) := \prod_{1\le i\le n} \GL_{d_i}(K)$ are in natural bijection with the isomorphism classes of the $\bd$-dimensional $\la$-modules. Given $\SS \in \Seq(\bd)$, we denote by $\laySS$ the locally closed subvariety of $\modlad$ which consists of the points $x$ for which the corresponding module $M_x$ has radical layering $\SS$.  Clearly, the varieties $\laySS$, where $\SS$ traces the semisimple sequences with $\udim \SS = \bd$, partition $\modlad$.  However, in general, this (finite) partition falls short of being a stratification of $\modlad$ in the strict sense, in that closures of strata need not be unions of strata.  

To introduce the projective parametrizing variety $\grassbd$, we fix a projective $\la$-module $\bP$ whose top $\bP/J\bP$ has dimension vector $\bd$, and set $d=|\bd|$. The variety $\grassbd$ is the closed subvariety of the vector space Grassmannian $\Gr\bigl( (\dim\bP - d), \bP \bigr)$ consisting of those points $C \in \Gr\bigl( (\dim\bP - d), \bP \bigr)$ which are $\la$-submodules of $\bP$ with the property that $\udim(\bP/C) = \bd$. This time, the group action whose orbits determine the isomorphism classes of the quotients $\bP/C$ in $\lamod$ is the canonical action of $\Aut_\la(\bP)$ on $\grassbd$. The role played by $\laySS$ in the affine setting is taken over by $\biggrassSS$, the locally closed subvariety consisting of those $C \in \grassbd$ for which $\SS(\bP/C) = \SS$. 

The following connection between the affine and projective parametrizing varieties was proved in \cite[Proposition C]{BoHZ}; it was inspired by Gabriel's \cite{Gab}, as is explained in some detail in Remark 3 of \cite[Section 2]{BoHZ}.  We restate the result for convenient reference. 

\begin{proposition} \label{prop2.1}
 Consider the natural isomorphism from the lattice of $\GL(\bd)$-stable subsets of $\modlad$ on one hand to the lattice of $\Aut_\la(\bP)$-stable subsets of $\grassbd$ on the other, which pairs orbits encoding  isomorphic modules. This correspondence  preserves and reflects openness, closures, irreducibility, and smoothness.
\end{proposition}

In describing generic projective resolutions of the modules in an irreducible component of $\modlad$, a key invariant of a $\bd$-dimensional $\la$-module $M$ is its set of skeleta. These skeleta live in a projective cover of $M$ in $\latruncmod$. In the following definitions, we fix a semisimple sequence $\SS$ with $\udim \SS = \bd$.

\begin{definitions}  \label{def2.2} \textbf{Coordinatized projective modules and skeleta.} 
\smallskip

\textbf{(1)} Let $\Ptrunc$ be a projective cover of $\SS_0$ in $\latruncmod$. This cover is referred to as a {\it coordinatized\/} projective module when it comes equipped with a fixed full sequence of top elements $\bz_1,\dots,\bz_t$, where $t= \dim \SS_0$.  
In particular, we obtain a decomposition $\Ptrunc = \bigoplus_{1\le r\le t} \latrunc\, \bz_r$. 
A {\it path of length $l$\/} in the coordinatized projective module $\Ptrunc$ is any nonzero element $\bp = p\,\bz_r$ where $p$ is a path of length $l$ in $Q$;  thus each $\bz_r$ is now viewed as a path of length zero. Note that we have a well-defined concept of path length in $\latrunc$, and hence also in $\Ptrunc$.  Clearly, each path $\bp = p\, \bz_r \in \Ptrunc$  is normed by a primitive idempotent, namely by $\term(p)$, and the primitive idempotent norming $\bz_r$ is $\strt(p)$.
\smallskip

\textbf{(2)} An (abstract) {\it skeleton with layering $\SS$\/} is a set $\S$ consisting of paths in $\Ptrunc$ which satisfies the following two conditions:  
\begin{enumerate}
\item It is closed under initial subpaths, i.e., whenever $p\, \bz_r \in \S$, and $q$ is an initial subpath of $p$ (meaning $p = q' q$ for some path $q'$), the path $q\, \bz_r$ again belongs to $\S$. 
\item For $0 \le l \le L$, the number of those paths of length $l$ in $\S$ which end in a given vertex $e_i$ coincides with the multiplicity of $S_i$ in the semisimple module $\SS_l$.
\end{enumerate}
Note that any skeleton $\S$ with layering $\SS$ includes the paths $\bz_1, \dots, \bz_t$ of length zero.  
\smallskip

\textbf{(3)} Let $M \in \lamod$.  An abstract skeleton $\S$ is a {\it skeleton of\/} $M$ in case $M$ has a full sequence $z_1, \dots, z_t$ of top elements, each $z_r$ normed by the same vertex as $\bz_r$, such that  
\begin{enumerate}
\item $\{ p\, z_r \mid p\, \bz_r \in \S\}$ is a $K$-basis for $M$, and 
\item the layering of $\S$ coincides with the radical layering $\SS(M)$ of $M$.  
\end{enumerate}
In this situation, we also say that $\S$ is a skeleton of $M$ {\it relative to\/} $z_1,\dots,z_t$.
\end{definitions}

Clearly, the set of skeleta of any finite dimensional $\la$-module $M$ is non-empty, and the set of all skeleta of modules with fixed dimension vector $\bd$ is finite.  The relevance of skeleta towards a generic understanding of the modules in the irreducible components of $\modlad$ is underlined by the following fact:  
%%The map which sends any point $x \in \modlad$ to the set of skeleta of $M_x$ is generically constant on irreducible components, as follows.

\begin{observation} Let $\P$ be the power set of the set of all skeleta with dimension vector $\bd$. Then the map
$$\modlad \longrightarrow \P, \qquad x \longmapsto \{ \text{skeleta of}\ M_x \}$$
is generically constant on each irreducible component of $\modlad$.

To see this, let $\C \subseteq \modlad$ be an irreducible component, and $\SS$ the generic radical layering of its modules. Then $\C \cap \laySS$ is open in $\C$, and for any skeleton $\S$ with layering $\SS$, the set
$$\Rep(\S) := \{ x \in \modlad \mid \S \ \text{is a sketon of} \ M_x \}$$
is an open subvariety of $\laySS$; see \cite[Lemma 3.8]{classifying}. Hence, a skeleton $\S$ with layering $\SS$ arises as a skeleton of the modules in a dense open subset of $\C$ precisely when $\C \cap \Rep(\S)$ is nonempty. Given that there are only finitely many eligible skeleta, this proves the claim.
\end{observation}

Next, we recall more discerning graphical invariants associated to a finite dimensional $\la$-module, namely its hypergraphs; see \cite[Definition 3.9]{BHT}. 
%\medskip

\begin{definitions}  \label{def2.3} \textbf{$\S$-critical paths and hypergraphs.} 
Again, we let $\Ptrunc$ be a coordinatized projective $\latrunc$-module with top $\SS_0$ and assume $\S \subseteq \Ptrunc$ to be an abstract skeleton with layering $\SS$.  Recall that the distinguished top elements $\bz_r$ of $\Ptrunc$ coincide with the paths of length zero in $\S$.
\smallskip

\textbf{(1)} A {\it $\sigma$-critical path\/} is a path $\bq \in \Ptrunc \setminus \sigma$ such that every \underbar{proper} initial subpath of $\bq$ belongs to $\S$. Thus, $\bq = \alpha \bq'$ where $\bq' \in \sigma$ and $\alpha$ is an arrow; in particular, $\len(\bq) > 0$.   Given a $\sigma$-critical path  $\bq$, we define a subset $\S_\bq \subseteq \S$ as follows: 
$$\S_{\bq} := \{ \text{paths\ }\bp \in \sigma \mid \len(\bp) \ge \len(\bq) \,\ \text{and} \ \term(\bp) = \term(\bq) \}.$$
The final condition in the definition of $\sigma_{\bq}$ means that all paths in $\sigma_{\bq}$ are normed (on the left) by the same vertex as $\bq$.
\smallskip

\textbf{(2)} Suppose $M \in \lamod$ has skeleton $\S$ relative to a full sequence $z_1, \dots, z_t$ of top elements.  The $\la$-structure of $M$ is then determined by the family of expansion coefficients corresponding to the $\S$-critical paths $\bq = q\, \bz_r \in \Ptrunc$, namely
\begin{equation}  \label{expandqzr}
q\,z_r = \sum_{\bp = p \bz_s  \in \S_\bq} c_{\bq, \bp}\, p\, z_s
\end{equation}
for unique scalars $c_{\bp,\bq} \in K$.
\smallskip

\textbf{(3)} We refer to any pair 
 $$\G = (\sigma, \bigl(\tau_\bq)_{\bq \; \sigma\text{-critical}}\bigr)  \ \ \ \text{with} \ \ \tau_{\bq} \subseteq \sigma_{\bq} \ \ \text{for all} \ \S\text{-critical paths}\ \bq$$ 
 as an (undirected) {\it hypergraph in\/} $\Ptrunc$. The set $\tau_\bq$ is called the {\it support set\/} of $\bq$. Empty support sets are allowed.  
 
 In informal terms: The vertices of these hypergraphs are the elements of $\sigma$, and a typical (hyper)edge, labeled by an arrow $\gamma \in Q_1$, connects a vertex $\bp \in \S$ to the vertex $\gamma \bp$ in case $\gamma \bp \in \S$ and to the support set $\tau_{\gamma\bp}$ of vertices if $\gamma\bp$ is $\S$-critical. 
\smallskip

\textbf{(4)} A hypergraph $\G$ as above is called a {\it hypergraph of\/} a $\la$-module $M$ 
  (relative to a  full sequence $z_1, \dots, z_t$ of top elements of $M$) if $\S$ is a skeleton of $M$ and, in the expansion \eqref{expandqzr} above, $c_{\bq,\bp} \ne 0$  precisely when $\bp \in \tau_\bq$.
\end{definitions}

While hypergraphs pin down \underbar{families} of modules, as opposed to individual isomorphism classes, they provide a useful tool for communicating, in a visually suggestive format, the generic structure of the modules in the components.  For our diagrammatic representations of hypergraphs, we refer to \cite{BHT}, \cite{DHW}, and to the example below.  This example will serve as a staple in the sequel.

\begin{example} \label{ex2.4}  Let $\la = KQ / \langle \text{the paths of length}\ 4 \rangle = \latrunc$, where $Q$ is the quiver
$$\xymatrixrowsep{1.5pc}\xymatrixcolsep{6pc}
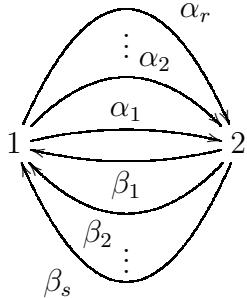
\xymatrix{
1 \ar@/^/[r]^{\alpha_1}  \ar@/^5ex/[r]^(0.6){\alpha_2} \ar@/^10ex/[r]^(0.7){\alpha_r}_{\vdots}  &2
\ar@/^/[l]^{\beta_1}  \ar@/^5ex/[l]^(0.6){\beta_2} \ar@/^10ex/[l]^(0.7){\beta_s}_{\vdots}
}$$

\textbf{(a)}  First suppose that $r = 2$ and $s= 1$.  Choose $\SS := (S_1, S_2, S_1, S_2)$, and let $\Ptrunc = \latrunc \bz$ be the corresponding $\latrunc$-projective cover of $\SS_0 = S_1$, coordinatized by a fixed top element $\bz$. 
Generically, the modules in $\Rep \SS$ then have a hypergraph of the form
$$\xymatrixrowsep{1.75pc}\xymatrixcolsep{1.5pc}
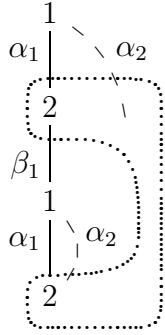
\xymatrix{
1 \edge[d]_(0.4){\alpha_1}  \\
2 \edge[d]_(0.6){\beta_1} & \save+<0ex,-0.1ex> \dashedge@/_/[ul]_(0.4){\alpha_2} \restore  \\
1 \edge[d]_(0.4){\alpha_1} \dashedge@/^2ex/[d]^(0.4){\alpha_2}   &&  \\
2 \save[0,0]+(0,3);[0,0]+(0,3) **\crv{~*=<2.5pt>{.} [0,1]+(2,3) &[-1,1]+(2,-4) &[-2,1]+(2,-4) &[-2,0]+(0,-4) &[-2,0]+(-3,-4) &[-2,0]+(-3,0) &[-2,0]+(-3,4) &[-2,0]+(0,4) &[-2,1]+(0,4) &[-2,1]+(5,4) &[-2,1]+(5,0) &[-2,1]+(5,-4) &[-1,1]+(5,0) &[0,1]+(5,0) &[0,1]+(5,-4) &[0,1]+(0,-4) &[0,0]+(0,-4) &[0,0]+(-3,-4) &[0,0]+(-3,0) &[0,0]+(-3,3)}\restore &&
}$$
\medskip

\noindent This diagram is to be read as follows: The radical layering of any module $G$ having the above hypergraph (relative to a top element $z \in G$, say) is $\SS$, and the skeleton chosen to represent $G$ is $\S := \{\bz, \alpha_1 \bz, \beta_1 \alpha_1 \bz, \alpha_1 \beta_1 \alpha_1 \bz\}$; the edges corresponding to paths in the skeleton $\S$ are drawn as solid edges, while the dashed edges stand for the terminal arrows of $\S$-critical paths.  Moreover, the diagram contains the information that the support sets $\tau_{\bq}$ for the two $\S$-critical paths $\bq = \alpha_2  \bz$ and $\bq = \alpha_2 \beta_1 \alpha_1 \bz$ in $\Ptrunc$ (in the sense of Definition \ref{def2.3}), are $\tau_{\alpha_2 \bz} = \{ \alpha_1 \bz,\, \alpha_1 \beta_1 \alpha_1 \bz\}$ and $\tau_{\alpha_2 \beta_1 \alpha_1 \bz} = \{ \alpha_1 \beta_1 \alpha_1 \bz\}$.  Indeed, the ``dotted pool" indicates that the element $\alpha_2 z$ of $G$ is a $K$-linear combination of $\alpha_1 z$ and $\alpha_1 \beta_1 \alpha_1 z$ with coefficients in $K^*$; on the other hand, given that the set $\tau_{\alpha_2 \beta_1 \alpha_1 \bz}$ is a singleton, no extra pooling device is required to communicate the condition that $\alpha_2 \beta_1 \alpha_1 z \in G$ be a nonzero scalar multiple of $\alpha_1 \beta_1 \alpha_1 z$. 

Next, we consider the semisimple sequence $\SS' := (S_1^2, S_2^2, 0,0)$.  The modules in $\laySS'$ generically look as follows, relative to top elements $z_1, z_2$ say:
$$\xymatrixrowsep{2pc}\xymatrixcolsep{1.5pc}
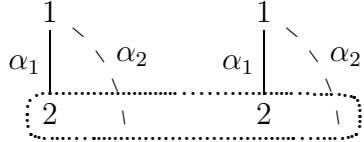
\xymatrix{
1 \edge[d]_{\alpha_1}  &&&1 \edge[d]_{\alpha_1}  \\
2 \horizpool{4} & \save+<0ex,-0.1ex> \dashedge@/_/[ul]_(0.4){\alpha_2} \restore &&2 & \save+<0ex,-0.1ex> \dashedge@/_/[ul]_(0.4){\alpha_2} \restore 
}$$
\smallskip

\noindent Here, the dotted pool serves double duty in indicating that both $\alpha_2 z_1$ and $\alpha_2 z_2$ are linear combinations of $\alpha_1 z_1$ and $\alpha_1 z_2$ with (unspecified) nonzero coefficients.  In the sequel, we will use the fact that, generically, the modules  in $\laySS'$ decompose in the form
$$\xymatrixrowsep{2pc}\xymatrixcolsep{1pc}
\xymatrix{
1 \edge[d]_{\alpha_1} \dashedge@/^2ex/[d]^{\alpha_2}  &&\ar@{}[d]|{\;\;\textstyle\bigoplus} &&1 \edge[d]_{\alpha_1} \dashedge@/^2ex/[d]^{\alpha_2}  \\
2 &&&&2
}$$
\medskip

\textbf{(b)}  Now let $r = 3$. The hypergraphs
$$\xymatrixrowsep{2pc}\xymatrixcolsep{1.5pc}
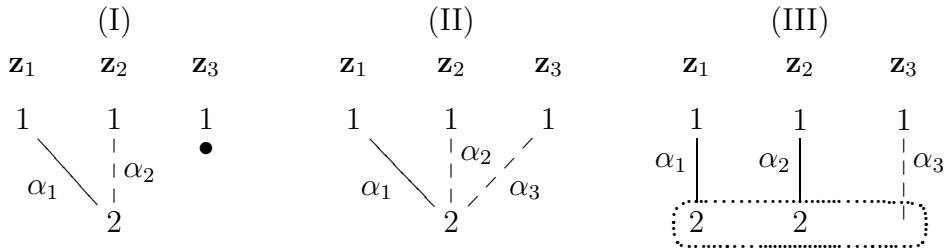
\xymatrix{
&\txt{(I)} & &&&\txt{(II)} & &&&\txt{(III)}  \\
1 \dropup{\bz_1} \edge[dr]_{\alpha_1} &1 \dropup{\bz_2} \dashedge[d]^{\alpha_2} &1 \dropup{\bz_3} \drbl  &&1 \dropup{\bz_1} \edge[dr]_{\alpha_1} &1 \dropup{\bz_2} \dashedge[d]^(0.33){\alpha_2} &1 \dropup{\bz_3} \dashedge[dl]^{\alpha_3}  &&1 \dropup{\bz_1} \edge[d]_(0.4){\alpha_1} &1 \dropup{\bz_2} \edge[d]_(0.4){\alpha_2} &1 \dropup{\bz_3} \dashedge[d]^(0.4){\alpha_3}  \\
&2 & &&&2 & &&2 \horizpool{2} &2 &
}$$
\smallskip

\noindent are hypergraphs of modules $M_i  = (\bigoplus_{1 \le j \le 3} \la \bz_j) / U_i$, where  $\bz_j = e_1$ for $j = 1,2,3$.  Here the submodule $U_1$ is generated by $\alpha_2 \bz_2 - \alpha_1 \bz_1$, $\alpha_3 \bz_3$ and $\alpha_j \bz_k$ for $j \ne k$, while  $U_2$ is generated by $\alpha_2 \bz_2 - \alpha_1 \bz_1$, $\alpha_3 \bz_3 - \alpha_1 \bz_1$ and $\alpha_j \bz_k$ for $j \ne k$; finally, $U_3$ is generated by $\alpha_3 \bz_3 - (\alpha_1 \bz_1 + \alpha_2 \bz_2)$ and $\alpha_j \bz_k$ for $j \ne k$. The chosen reference skeleton of $M_1$ and $M_2$ is $\S := \{\bz_1, \bz_2, \bz_3, \alpha_1 \bz_1\}$, and that of $M_3$ is $\S \cup\{\alpha_2 \bz_2\}$.  Note that the dimension of $JM_3$ is $2$, the number of displayed vertices in the second row of the hypergraph.
\smallskip

Generically, the modules with radical layering $\SS' := (S_1^2, S_2^2, 0, 0)$ are indecomposable and have hypergraphs of the form
$$\xymatrixrowsep{2pc}\xymatrixcolsep{1.5pc}
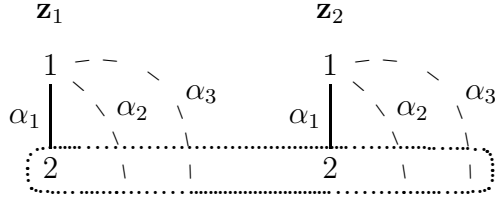
\xymatrix{
1 \dropup{\bz_1} \edge[d]_{\alpha_1}  &&&&1 \dropup{\bz_2} \edge[d]_{\alpha_1}  \\
2 \horizpool{6} & \save+<0ex,-0.1ex> \dashedge@/_/[ul]_(0.4){\alpha_2} \restore & \save+<0ex,-0.1ex> \dashedge@/_4ex/[ull]_(0.3){\alpha_3} \restore &&2 & \save+<0ex,-0.1ex> \dashedge@/_/[ul]_(0.4){\alpha_2} \restore & \save+<0ex,-0.1ex> \dashedge@/_4ex/[ull]_(0.3){\alpha_3} \restore
}$$
\smallskip

The modules in $\laySS$, where $\SS := (S_1, S_2, S_1, S_2)$, generically have a hypergraph akin to the first one shown in part (a).
\qed
\end{example}

%%%%%%%%%%%%%%%%%%%%%

\section{The main results for general $\la$}
\label{sec3}

%%%%%%%%%
\subsection*{3.A.~Pared-down parametrizing varieties}  \hfill\par
\smallskip

Towards a description of $\overline{\laySS}$, we present lower-dimensional, more manageable varieties parametrizing the modules with radical layering $\SS$. 

\begin{definition}  \label{def3.1} \textbf{Decompositions of $K^{|\bd|}$ induced by semisimple sequences.}
Let $\SS =  (\SS_0, \dots, \SS_L)$ be a realizable semisimple sequence in $\lamod$ with $\udim \SS = \bd$, and write $d = |\bd|$.
Consider a vector space decomposition of $K^d$ which is {\it induced by\/} $\SS$ in the following sense: Namely,
$$K^d \ = \bigoplus_{0 \le l \le L,\, 1 \le i \le n} \K_{(l,i)}$$
with the property that $\dim  \K_{(l,i)} = \dim e_i \SS_l$ for all eligible indices $l$ and $i$.  Set $\K_l = \bigoplus_{1 \le i \le n} \K_{(l,i)}$ for $l \le L$, and $\K_{L+1} = \K_{(L+1,i)} = 0$. 
Given a family $(f_{\alpha})_{\alpha \in Q_1}$ of $K$-endomorphisms of $K^d$, the following notation will be convenient: Whenever $p = \alpha_l \cdots \alpha_1$ is a path of positive length $l$ in $Q$, we set $f_p = f_{\alpha_l} \circ \cdots \circ f_{\alpha_1}$; if $p$ is a path of length $0$, say $p = e_i$, then $f_p$ is defined to be the canonical projection 
$K^d \rightarrow   \bigoplus_{0 \le l \le L} \K_{(l,i)} \subseteq K^d$ relative to the above decomposition. Thus, we obtain a $K$-algebra homomorphism $KQ \rightarrow \End_K(K^d)$ such that $p \mapsto f_p$ for all paths $p$ in $Q$.  
\end{definition}

By  $Q_{\ge l}$ we denote the set of paths of length at least $l$ in $Q$. 
The following lemma is an upgraded version of \cite[Lemma 5.1]{irredcompI} and is proved analogously.   

\begin{lemma}  \label{lem3.2}  {\rm \textbf{Triangular points in $\modlad$.}}

\noindent   We refer to the above notation.   
Suppose that $f = \bigl( f_{\alpha} \bigr)_{\alpha \in Q_1}$ is a family of $K$-linear maps $K^d \rightarrow K^d$ satisfying the following three conditions:  For any arrow $\alpha$ from $e_i$ to $e_j$ and any index $l \in \{0, \dots, L\}$, 
\smallskip

{\rm (i)}\  $f_{\alpha} (\K_{(l, r)}) = 0$ \  for all \ $r \ne i$;
\smallskip

{\rm (ii)}\  $f_{\alpha} (\K_{(l, i)}) \subseteq \bigoplus_{l+1 \le m \le L} \K_{(m, j)}$; 
\smallskip

{\rm (iii)}\  whenever $c_1, \dots, c_m \in K$ and $p_1,\dots, p_m$ are paths of length $\le L$ in $Q$, which have a common starting vertex and a common terminal vertex,  
$$\sum_{1 \le j \le m} c_j p_j \in I\  \  \implies \  \  \sum_{1 \le j \le m} c_j f_{p_j } = 0.$$ 
  
\medskip

\noindent Then the following statements {\bf (I)} -- {\bf (III)} hold:
\smallskip

\noindent {\rm \textbf{(I)}} The tuple $f$ is a point in $\modlad$, and the radical layering of the corresponding $\la$-module $M_f$ satisfies $\, \SS(M_f) \ge \SS$.   Moreover, all $\la$-modules with radical layering $\SS$ are represented by suitable points $f \in \modlad$ satisfying {\rm (i)} -- {\rm (iii)}. 
\smallskip 
 
\noindent {\rm \textbf{(II)}} $J^l M_f =   \sum_{p \in Q_{\ge l}} \Img(f_p)$ for all $l \in \{0, \dots, L\}$.
  \smallskip

\noindent {\rm \textbf{(III)}} $\SS(M_f) = \SS$ precisely when, for each $h \in \{0, \dots, L\}$, the linear map 
$$(\K_0)^{Q_{\ge h}}  \rightarrow \bigoplus_{l \ge h} \K_l, \ \ \ (x_q)_{q \in Q_{\ge h}}\ \mapsto\ \sum_{q} f_q(x_q)$$ 
has maximal rank, namely $\sum_{l \ge h} \dim \K_l$. 
\qed\end{lemma}

The lemma prompts an analysis of the following two subvarieties of $\modlad$.

\begin{subsec}  \label{def3.3}  \textbf{The varieties $\scaledrep(\ge \SS)$ and $\scaledrep \SS$.}  Keep $\SS$ and a decomposition of $K^d$ induced by $\SS$ fixed. The collection of all $f = (f_\alpha)$ satisfying conditions (i) -- (iii) of Lemma \ref{lem3.2} is a closed subvariety of $\modlad$ which we denote by $\scaledrep({\ge} \SS)$.   Indeed, the inclusion map 
$$\scaledrep({\ge} \SS) \hookrightarrow \modlad$$
provided by part (I) of Lemma \ref{lem3.2} is a closed immersion.  

To see this, take $B_{(l,\mu)}  = \bigl(b_{(l, \mu)}^{1}, \dots, b_{(l, \mu)}^{d_{l,\mu}} \bigr)$ to be an ordered basis for $\K_{(l,\mu)}$ and $B$ to be the lexicographically ordered union of the $B_{(l,\mu)}$.  Relative to this basis for $K^d$, the image of the above embedding consists of all those families $(F_\alpha)$ of matrices in $\modlad$ such that each $F_\alpha$ has a strictly lower triangular form of the following ilk:  $\bullet$ The only nonzero entries in any column labeled $(l, \mu)^{(j)}$ are confined to positions with lower label $(l+1, \nu), \dots, (L, \nu)$, provided $\alpha$ is an arrow $e_\mu \rightarrow e_\nu$, and  $\bullet$ condition (iii) of Lemma \ref{lem3.2} is satisfied. The latter requirement translates into polynomial equations for the entries of the $F_\alpha$.  This shows that the considered embedding is indeed a closed immersion. 

Observe moreover that, up to isomorphism, the variety $\scaledrep ({\ge} \SS)$ is determined by $\SS$, irrespective of the choice of a decomposition  $K^d = \bigoplus_{l,i} \K_{(l,i)}$ induced by $\SS$.  Lemma \ref{lem3.6} below will show that the $\GL(\bd)$-stable hull $\GL(\bd).\bigl( \scaledrep ({\ge} \SS) \bigr) \subseteq \modlad$ is, in fact, unique in the strict sense.  

We will identify $\scaledrep ({\ge} \SS)$ with its image under the above immersion whenever convenient. The subset of $\scaledrep({\ge} \SS)$ consisting of the points  which correspond to modules with radical layering $\SS$ will be denoted by $\scaledrep \SS$.   In view of part (III) of Lemma \ref{lem3.2}, $\scaledrep \SS$ is an open subvariety of $\scaledrep ({\ge} \SS)$. 
\end{subsec}

Next, we consider the effect of conjugation by $\GL(\bd)$ on the varieties $\scaledrep ({\ge} \SS)$ and $\scaledrep (\SS)$.   

\begin{subsec}  \label{def3.4}  \textbf{$\scaledrep ({\ge} \SS)$ under the $\GL(\bd)$-action.}  
Viewed as  subvarieties of $\modlad$, the varieties  $\scaledrep( {\ge} \SS)$ and $\scaledrep (\SS)$ fail to be stable under the $\GL(\bd)$-action in all nontrivial cases.  However, each of these varieties carries a conjugation action by the subgroup $\GL(\SS)$ of $\GL(\bd)$ which consists of the sequences $(g_1,\dots,g_n)$ with the property that each $g_i$ leaves the subspaces $\bigoplus_{j \ge l} \K_{(j,i)}$ invariant for all $l$.  Caveat:  The $\GL(\SS)$-action does not separate the isomorphism classes of the pertinent modules in general.  

By part (I) of Lemma \ref{lem3.2}, the closure of $\scaledrep ({\ge} \SS)$ under the $\GL(\bd)$-action on $\modlad$ is contained in the closed subvariety $\bigcup_{\SS' \ge \SS} \Rep \SS'$ of $\modlad$.  In fact, in view of the lemma,
$${\laySS} \ \ =\ \ \GL(\bd). \bigl(\scaledrep (\SS)\bigr) \ \ \subseteq \ \ \GL(\bd). \bigl(\scaledrep ({\ge} \SS)\bigr) \ \ \subseteq \ \ \bigcup_{\SS' \ge \SS} \Rep \SS'.$$  
Either inclusion may be proper.  This is obvious for the first.  Regarding the second, let $\la = K Q/ \langle \beta^2 \rangle$, for instance, where $Q :={ \xymatrixcolsep{1.5pc} \xymatrix{ 1 \ar[r]^{\alpha} &2 \ar@(ur,dr)^{\beta} }}$.  Moreover, take $\SS := (S_1^2, S_2^2)$ and $\SStilde := (S_1^2 \oplus S_2, S_2)$. Then $\SStilde \ge \SS$, but the module $N := S_1^2 \oplus \la e_2$ in $\Rep(\SStilde)$ is not isomorphic to a module in $\scaledrep({\ge} \SS)$.  Indeed, since $\K_{(0,2)} = 0$ and $\dim \K_{(1,2)} = 2$ in the decomposition of $K^4$ induced by $\SS$, we have $S_2^2 \subseteq \soc M$ for all $M$ in $\scaledrep(\ge\SS)$, while this is not the case for $N$.
\end{subsec}
%\medskip

%%%%%%%%%%
\subsection*{3.B.~The closure of $\laySS$ in $\modlad$}  \hfill\par
\smallskip

We start with an elementary lemma characterizing the modules corresponding to the points in $\scaledrep ({\ge} \SS)$.  
For a given realizable semisimple sequence $\SS = (\SS_0, \dots, \SS_L)$ with $\udim \SS = \bd$, we fix a decomposition of $K^{|\bd|}$ induced by $\SS$ as in Definition \ref{def3.1}.  As we already pointed out, modulo isomorphism of varieties, this choice has no bearing on $\scaledrep ({\ge} \SS)$.  

\begin{definition}  \label{def3.5}  \textbf{Filtrations governed by $\SS$.}  Let $M$ be a $\la$-module.  A {\it filtration of\/} $M$ {\it governed by\/} $\SS$ is any chain of submodules
$$M = M_0 \supseteq M_1 \supseteq \cdots \supseteq M_{L+1} = 0$$
such that each factor $M_l/M_{l+1}$ is isomorphic to $\SS_l$; in other words,
$J M_l \subseteq M_{l+1}$ and $\udim M_l/M_{l+1} = \udim \SS_l$ for $0 \le l \le L$.  Filtrations with these properties  will also be referred to more briefly as {\it $\SS$-filtrations\/}. 
\end{definition}  

\begin{lemma-definition}  \label{lem3.6} {\rm\textbf{The variety $\Filt \SS$.}}  Let $\la = KQ/I$ be an arbitrary basic finite dimensional $K$-algebra.  Moreover, let $\SS$ be a semisimple sequence with $\udim \SS = \bd$.  Then the following conditions are equivalent for a $\la$-module $M$:
\begin{enumerate}
\item $M$ belongs to the $\GL(\bd)$-stable hull of $\scaledrep(\ge \SS)$, that is, to $\GL(\bd).\bigl(\scaledrep ({\ge} \SS) \bigr)$.
\item $M$ has a filtration governed by $\SS$. 
\end{enumerate}

In particular, $\GL(\bd).\bigl(\scaledrep ({\ge} \SS) \bigr)$ is independent of the choice of a decomposition of $K^{|\bd|}$ induced by $\SS$.  Motivated by the above equivalence, we will denote this subvariety of $\modlad$ by $\Filt \SS$. 
\end{lemma-definition}

\begin{proof} (1)$\implies$(2):  Suppose that $M$ is represented by some point $f = (f_\alpha) \in \scaledrep ({\ge} \SS)$.  This means that, up to isomorphism, $M$ equals $K^d$, equipped with the $\la$-module structure of Lemma \ref{lem3.2}.  In particular, we obtain a filtration of $M$ governed by $\SS$ by setting $M_l =  \bigoplus_{j \ge l,\, 1 \le i \le n} \K_{(j,i)}$.

(2)$\implies$(1):   Given an $\SS$-filtration $(M_l)_{0 \le l \le L+1}$ of $M$, we take $M_{(l,i)}$ to be a vector space complement of $e_i M_{l+1}$ in $e_i M_l$ for $0 \le l  \le L$.  Moreover, we set $f = (f_\alpha)_{\alpha \in Q_1}$, where $f_\alpha(x) = \alpha x$ for $x \in M$.  Then the decomposition $M = \bigoplus_{0 \le l \le L,\, 1 \le i \le n} M_{(l,i)}$ satisfies conditions (i) -- (iii) of Lemma \ref{lem3.2}, and thus can be shifted to a decomposition $\bigoplus_{0 \le l \le L,\, 1 \le i \le n} \K_{(l,i)}$ of $K^d$ induced by $\SS$ via a suitable family $h = (h_{(l,i)})$ of isomorphisms  $h_{(l,i)}: M_{(l,i)} \rightarrow \K_{(l,i)}$.  We conclude $h f h^{-1} \in \scaledrep ({\ge} \SS)$ and $M_{h f h^{-1}} \cong M$. 
\end{proof} 

The upcoming remarks (1)--(3) will be tacitly invested in the sequel. 

\begin{remarks}  \label{rem3.7}
\textbf{(1)} $\Filt \SS$ is always nonempty, irrespective of whether $\SS$ is realizable.  Indeed, the semisimple module $\bigoplus_{0\le l\le L} \SS_l$ has a filtration governed by $\SS$.  

\textbf{(2)} For any $M \in \lamod$, the chain $M \supseteq JM \supseteq \cdots \supseteq J^{L+1}M = 0$ is the only filtration of $M$ governed by $\SS(M)$; moreover,
 if $\SS'$ is any semisimple sequence governing a filtration of $M$, then $\SS' \le \SS(M)$.  
 
\textbf{(3)} The socle layering $\SS^*(M)$ of $M$ governs the socle filtration, provided the traditional indexing of the latter is reversed; i.e., if $\SS^*(M) = (\SS^*_0, \dots, \SS^*_m, 0 , \dots, 0)$ with $\SS^*_m \ne 0$, then the filtration $\soc_m M = M \supseteq \soc_{m-1} M \supseteq \cdots \supseteq \soc_0 M = \soc M \supseteq 0$ is governed by the (not necessarily realizable) semisimple sequence $(\SS^*_m, \dots, \SS^*_0, 0 , \dots, 0)$.  In particular, $(\SS^*_m, \dots, \SS^*_0, 0, \dots, 0) \le \SS(M)$.    

\textbf{(4)} K. Bongartz pointed out to us that the upcoming Theorem \ref{thm3.8} may alternatively be derived from a useful result of Steinberg.  We state it below, but omit detail.  We do fully anchor our own steppingstone to \ref{thm3.8} (namely Theorem \ref{prop3.9}) though.  The embedding of $\biggrassSS$ into a flag variety, as specified there, is instrumental in a further analysis of the closure of $\biggrassSS$ in $\grassbd$.

Lemma \cite[Lemma 2, p.68]{Ste}: Let $V$ be a quasi-projective variety carrying a morphic action by a connected linear algebraic group $G$. Moreover, let $U$ be a closed subvariety of $V$ which is stable under the action of some parabolic subgroup of $G$. Then the $G$-stable hull $G.U$ of $U$ in $V$ is in turn closed.
 \end{remarks}

\begin{theorem}  \label{thm3.8}  Let $\la$ be an arbitrary basic finite dimensional algebra, and let $\SS$ be a semisimple sequence in $\lamod$ with $\udim \SS = \bd$.
Then the $\GL(\bd)$-stable set $\Filt \SS$, which consists of the points in $\modlad$ encoding modules with $\SS$-filtrations, is a closed subvariety of $\modlad$.  
\smallskip

In particular, $\overline{\laySS} \subseteq \Filt \SS$, meaning that every module in $\overline{\laySS}$ has an $\SS$-filtration. 
\end{theorem}

To prove Theorem \ref{thm3.8}, we switch back and forth between the affine and projective settings, $\modlad$ and $\grassbd$, using Proposition \ref{prop2.1} to transfer information from one to the other.  Again, we denote by $\bP$ the $\la$-projective cover of $\bigoplus_{1 \le i \le n} S_i^{d_i}$ in whose submodule lattice the points of $\grassbd$ are located.  We start by establishing a natural embedding of $\biggrassSS$ into a projective variety consisting of submodule flags $D_{L+1} \subseteq D_L \subseteq \cdots \subseteq  D_0 = \bP$ of $\bP$ which are governed by $\SS$.  It is this embedding which makes information about the closure of $\biggrassSS$ in $\grassbd$ more accessible.
   
\begin{theorem} \label{prop3.9}  Consider the subset $\U$ of the partial flag variety $\flag(\partial_0,\dots,\partial_{L+1},\bP)$ of $\bP$, where $\partial_i := (\dim\bP - |\bd|) + \sum_{l= L+1-i}^L |\udim \SS_l|$, consisting of the $\la$-submodule flags 
$$0 \subseteq D_{L+1} \subseteq D_L \subseteq \cdots \subseteq  D_0 = \bP \ \ \text{with} \ \  D_l/ D_{l+1}  \cong \SS_l \ \ \text{for} \ \ 0 \le l \le L.$$
Then $\U$ is closed, and there is a natural embedding of varieties
$$\Phi:  \biggrassSS \rightarrow \U,$$
which induces an isomorphism onto its image.
\end{theorem}

\begin{proof}[Proof of {\rm\ref{prop3.9}}]  Recall that a module $N$ belongs to $\biggrassSS$, meaning that $N  \cong \bP/C$ with $C \in \biggrassSS$, precisely when $\udim \SS_l = \udim J^l N / J^{l+1} N = \udim (C + J^l \bP)/ (C + J^{l+1} \bP)$ for all eligible $l$.  Set $\bd^{(L+1)}  = \bd$ and $\bd^{(l)} = \bd - \sum_{l \le r \le L} \udim \SS_r$ for $0 \le l \le L$.  In particular, we obtain $\biggrass_{\bd^{(L+1)}} (\la) =  \grassbd$, and $\biggrass_{\bd^{(0)}} (\la) = \{\bP\}$.  

Clearly, $\U$ is a subset of the projective variety 
$$\biggrass_{\bd^{(L+1)}} (\la)\, \times\, \biggrass_{\bd^{(L)}} (\la)\, \times\,  \cdots \times\, \biggrass_{\bd^{(0)}} (\la);$$ 
namely, $\U$ consists of those points $(D_{L+1} , \dots, D_0)$ in the direct product that correspond to flags 
$D_{L+1} \subseteq D_L \subseteq \cdots \subseteq D_0 = \bP$
of $\la$-submodules of $\bP$ satisfying 
\begin{equation*} \tag{$\ddagger$}
J D_l \subseteq D_{l+1} \ \ \ \text{and} \ \ \  \udim D_l / D_{l+1} = \udim \SS_l \ \ \ \text{for} \ \ \ 0 \le l \le L.
\end{equation*} 
To verify that the set $\U$ is closed in the given direct product of module Grassmannians, note that the equalities under $(\ddagger)$, specifying the dimension vectors  of the consecutive quotients $D_l/ D_{l+1}$, are actually automatic;  this is due to the placement of the $D_l$ in $\biggrass_{\bd^{(l)}} (\la)$, respectively.  As for the inclusions under $(\ddagger)$:  It is well-known that, given any $f \in \End_K(\bP)$, the requirement ``$f(D_{l}) \subseteq D_{l+1}$ for all $l$" cuts a closed subset out of the variety
$$\{(D_l) \in  \prod_{0 \le l \le L+1} \biggrass_{\bd^{(l)}} (\la) \mid D_{l+1} \subseteq D_{l} \ \text{ for}\ 0 \le l \le L\}$$
of partial submodule flags. Applying this to the linear maps $\bP \rightarrow \bP$ given by $x \mapsto \alpha x$ for $\alpha \in Q_1$, and investing the fact that the displayed partial flag variety is closed in the given product of Grassmannians, one finds that $\U$ is indeed closed.  In particular, $\U$ is a projective variety.

We have a natural embedding of $\biggrassSS$ into $\U$, namely
$$\Phi: \biggrassSS \rightarrow \U, \ \ \ C \mapsto (C +J^{L+1} \bP ,\, C + J^L \bP,\, \dots,\, C + J \bP,\, C + J^0 \bP),$$
where the leftmost entry $C + J^{L+1} \bP$ of the sequence equals $C$, and the rightmost entry equals $\bP$.  

To see that $\Phi$ is a morphism, we use the open affine cover $(\biggrassS)_{\S}$ of $\biggrassSS$, where $\S$ traces the skeleta with layering $\SS$ and $\biggrassS \ne \varnothing$.  For that purpose, recall the following description of $\biggrassS$ from \cite{hier}.  We view the $\la$-projective cover $P$ of $\SS_0$ as a direct summand of the projective cover $\bP = \bigoplus_{1 \le r \le |\bd|} \la \bz_r$ of $\bigoplus_{0 \le l \le L} \SS_l$, say $P = \bigoplus_{1 \le r \le t} \la \bz_r$.  On identifying the top elements $\bz_r$ of $P$ with those of $\Ptrunc$ (see \ref{def2.2}), we retrieve each of the subsets $\S $ of $\Ptrunc$ as a subset of $P$; as such, $\S$ consists of $|\bd|$ linearly independent elements of $\bP$.  Define $s := \dim \bP - |\bd|$, and let $\Schu(\S)$ be the big open Schubert cell of $\Gr( s, \bP)$ consisting of the vector space complements of the subspace $\bigoplus_{\bp \in \S} K \bp$ in $\bP$. Then $\biggrassS : = \biggrassSS \cap \Schu(\S)$ is open in $\biggrassSS$, and the union of the $\biggrassS$, with $\S$ as specified, equals $\biggrassSS$; cf.~\cite[Observation 3.6]{hier}.  By \cite[Theorem 3.17]{hier}, the $\biggrassS$ are affine; in fact, they can readily be realized as closed subsets of the $K$-space $\bigwedge^s \bP$ relative to the Pl\"ucker coordinates $[c_1 \wedge  \cdots \wedge c_s]$  of $\Schu(\S)$.  

 Hence it suffices to show that, for each such skeleton $\S$, the restriction $\Phi_{\S}$ of $\Phi$ to $\biggrassS$ is a morphism. For $0 \le j \le L$, let $\sigma_j$ be the set of all paths of length $j$ in $\sigma$. Enumerate the elements of $\S$ so that increasing indices correspond to weakly decreasing lengths. If $t_l := |\S_l| + \cdots + |\S_L|$, we thus obtain $\bigsqcup_{\,l \le j \le L} \S_j$ in the form
 $$\bigsqcup_{l \le j \le L} \S_j =  \{\bp_1, \dots, \bp_{t_l}\} \ \ \ \text{for} \ \ 0 \le l \le L.$$
We deduce that, given any $K$-basis $c_1,\dots,c_s$ for a point $C \in \biggrassS$, the elements $c_1,\dots,c_s,\bp_1, \dots, \bp_{t_l}$ form a $K$-basis for $C+ J^l \bP$: Indeed, $J^l \bP$ is generated by the paths in $\bP$ of the form $q \bz_r$, where $q$ is a path of length $\ge l$ in $KQ$ and $r \le |\bd|$. Moreover, by the definition of $\biggrassS$, $\bp_1, \dots, \bp_{t_l}$ induce a basis for $J^l(\bP/C) = (J^l \bP + C)/C$. This shows that the restriction $\Phi_\S$ sends any point $C \in \biggrassS$ to 
$$\bigl(\, [c_1 \wedge  \cdots \wedge c_s], [c_1 \wedge  \cdots \wedge c_s \wedge \bp_1 \wedge \cdots \wedge \bp_{t_L}], \dots, [c_1 \wedge  \cdots \wedge c_s \wedge \bp_1 \wedge \cdots \wedge \bp_{t_0}] \,\bigr),$$
whence $\Phi_{\S}$ is indeed a morphism. 

Finally, we observe that $\Phi$ induces an isomorphism onto its image. Indeed, the inverse is the restriction  to $\Img(\Phi)$ of the projection onto the leftmost component of the direct product of the $\biggrass_{\bd^{(l)}}(\la)$, namely the restriction of
$$\Psi: \prod_{0 \le l \le L+1} \biggrass_{\bd^{(l)}} (\la) \rightarrow \grassbd, \ \ \ (D_{L+1}, \dots, D_0) \mapsto 
D_{L+1}$$
to $\Img(\Phi)$. Therefore $\Phi^{-1} : \Img(\Phi) \rightarrow \biggrassSS$ is a morphism.
\end{proof}      

\begin{proof}[Proof of {\rm\ref{thm3.8}}]  We refer to the notation in the proof of \ref{prop3.9}. 
Since $\U$ is a projective variety, so is $\Psi(\U)$.  In particular, $\Psi(\U)$ is closed in $\grassbd$.  

By condition $(\ddagger)$ spelled out in the proof of 3.9, the image $\Psi(\U) \subseteq \grassbd$ consists precisely of those points $C \in \grassbd$ which have the property that $\bP/C$ has a filtration governed by $\SS$; in particular $\Psi(\U)$ is stable under the $\Aut_\la(\bP)$-action of $\grassbd$.  In light of Lemma \ref{lem3.6}, Proposition \ref{prop2.1} thus matches up $\Psi(\U)$ with the $\GL(\bd)$-stable subset $\Filt \SS$ of $\modlad$ and tells us that $\Filt \SS$ is in turn closed.

For the final claim, it suffices to observe that $\laySS \subseteq \Filt \SS$.  
\end{proof}

Theorem \ref{thm3.8} prompts us to introduce a new module invariant which will turn out to be highly informative towards the detection of irreducible components of $\modlad$.
 
\begin{definition}  \label{def3.10} \textbf{The module invariant $\Gamma$.}
 For $M \in \lamod$, let $\Gamma(M)$ denote the number of realizable semisimple sequences which govern some filtration of $M$.
\end{definition}
 
\begin{corollary}  \label{cor3.11}  The map $\Gabull : \modlad \rightarrow \NN$, which sends $x$ to $\Gamma(M_x)$, is upper semicontinuous.
 
 In particular:  Whenever $\C$ is an irreducible component of some $\laySS$ such that $1 \in \Gabull(\C)$, the closure $\overline{\C}$ is an irreducible component of $\modlad$.
 \end{corollary}
 
 \begin{proof} Let $\R$ be the set of all realizable semisimple sequences with dimension vector $\bd$.  Moreover, for $a \in \NN$, let $\R(a)$ be the collection of all those intersections $\bigcap_{i} \Filt(\SS^{(i)})$ which involve at least $\,a\,$ distinct sequences $\SS^{(i)} \in \R$.  Then the pre-image $\Gabull^{-1}([a, \infty))$ is the union of the sets in $\R(a)$.  Since each $\Filt(\SS^{(i)})$ is closed in $\modlad$ by Theorem \ref{thm3.8} and $\R(a)$ is finite, the union $\Gabull^{-1}([a, \infty))$ is closed.  This proves the claim regarding upper semicontinuity.  
 
 To justify the final assertion, suppose that $\overline{\C}$ is properly contained in some irreducible component $\C'$ of $\modlad$.  Then $\C'$ is an irreducible component of some $\overline{\laySS'}$ with $\SS' < \SS$.  Since $\overline{\laySS'} \subseteq \Filt(\SS')$ by Theorem \ref{thm3.8}, all modules in $\overline{\C}$ have a filtration governed by $\SS'$ in this situation, whence $\Gamma(M) > 1$ for all $M \in \laySS$.  
 \end{proof}  
 
Now let $D= \Hom_K(-,K) : \lamod \rightarrow \text{mod-}\la$ be the standard duality.  Clearly,  $M \in \lamod$  contains a descending submodule chain governed by $\SS = (\SS_0, \dots, \SS_L)$ if and only if $D(M)$ contains an ascending chain  $M'_{-1} = 0 \subseteq M'_0 \subseteq \cdots \subseteq M'_L = D(M)$ which is {\it cogoverned} by $D(\SS) = \bigl(D(\SS_0),\dots, D(\SS_L)\bigr)$, in the sense that each of the consecutive quotients $M'_l /M'_{l-1}$ is isomorphic to $D(\SS_l)$.  We define $\Cofilt \SS'$ to be the subset of $\modlad$ whose points correspond to the modules which are cogoverned by a semisimple sequence $\SS'$.  The duality $\widehat{D}: \Rep_\bd(\lamod) \rightarrow \Rep_\bd(\text{mod-}\la)$ of \cite[Section 2.C]{irredcompII} thus yields the following dual of \ref{thm3.8}; we spell it out since, in size comparisons of $\overline{\C^{(i)}}$ versus $\overline{\C^{(j)}}$, for irreducible components $\C^{(k)}$ of $\Rep{\SS}$, one gains mileage in combining \ref{thm3.8} with its dual.  (Recall: The process of filtering the irreducible components of $\modlad$ out of $\{\overline{\C} \mid \C \ \text{is a component of some}\ \laySS \ \text{with} \ \udim \SS = \bd\}$ rests on comparisons of this ilk.) 
 
\begin{theorem}  \label{thm3.12}  {\rm\textbf{Dual of \ref{thm3.8}.}}  If $\SS^* = (\SS^*_0, \dots, \SS^*_L\bigr)$ is a semisimple sequence in $\lamod$ with dimension vector $\bd$, let $\Corep \SS^*$, resp.~$\Cofilt \SS^*$, be the set of all points in $\modlad$ which correspond to modules with socle series $\SS^*$, resp.~to modules with filtrations cogoverned by $\SS^*$.
  
  Then $\Cofilt(\SS^*)$ is a closed subvariety of $\modlad$, and consequently $\overline{\Corep \SS^*} \subseteq \Cofilt \SS^*$.  In particular:  If $\C$ is an irreducible component of $\laySS$ such that, generically, the modules in $\C$ have socle layering $\SS^*$, then  
 $\overline{\C} \subseteq \Filt \SS \cap \Cofilt \SS^*$.  \qed
 \end{theorem}

We close the section with an example to the effect that, in general, the inclusion $\overline{\laySS} \subseteq \Filt \SS$ may be proper and the final implication of Corollary \ref{cor3.11} need not be reversible.  This contrasts the situation where $\la = \latrunc$, as we will see in Section \ref{sec4}.  

\begin{example}  \label{ex3.13} Consider the quiver $Q$ of Example \ref{ex2.4} with $r=2$ and $s = 1$, and set 
$$\la = KQ/ \langle \beta_1 \alpha_2, \, \alpha_2 \beta_1 , \ \text{all paths of length} \ 4 \rangle.$$
Let $\bd := (2,2)$, $\SS := (S_1,S_2,S_1,S_2)$, and $\SS' := (S_1^2, S_2^2, 0, 0)$.  Then the varieties  $\laySS$ and $\laySS'$ are irreducible, and generically their modules  have hypergraphs
$$\xymatrixrowsep{1.75pc}\xymatrixcolsep{1.25pc}
\xymatrix{
1 \edge[d]_{\alpha_1} \dashedge@/^5ex/[ddd]^{\alpha_2}  &&&&&1 \edge@/_/[d]_{\alpha_1} \dashedge@/^/[d]^{\alpha_2} &&\ar@{}[d]|{\textstyle\bigoplus}  &&1 \edge@/_/[d]_{\alpha_1} \dashedge@/^/[d]^{\alpha_2}  \\
2 \edge[d]_{\beta_1}  &&&\txt{and}  &&2  &&&&2  \\
1 \edge[d]_{\alpha_1}  \\
2
}$$
respectively, whence both are contained in $\Filt \SS$. Clearly, $\overline{\laySS} \nsubseteq \overline{\laySS'}$, due to the generic Loewy lengths of the modules in $\laySS$ and $\laySS'$.  By comparing generic $\alpha_2$-ranks, one moreover finds that $\overline{\laySS'} \nsubseteq \overline{\laySS}$.  In conclusion, both $\overline{\laySS}$ and $\overline{\laySS'}$ are components of $\Filt \SS$.  In fact, both of these closures are even irreducible components of $\modlad$, the latter failing to satisfy the sufficient condition of Corollary \ref{cor3.11}.  Indeed, $\Gamma(M) = 2$ for all $M$ in $\overline{\laySS'}$.  

It is readily verified that the total number of components of $\modlad$ is three, the remaining component being $\overline{\laySS''} = \Filt(\SS'')$ for $\SS'' = (S_2, S_1, S_2, S_1)$.  
By contrast:  On replacing $\la$ by the associated truncated path algebra $\latrunc$, two of the three components of $\modlad$ fuse into a single component of $\Rep_{\bd}(\latrunc)$; see Example \ref{ex6.1}(b) below.   \qed 
\end{example}
\bigskip 

%%%%%%%%%%%%%%%%%%%%%%%%%

\section{The main results for truncated $\la$}
\label{sec4}

\emph{Throughout this section, $\la$ stands for a truncated path algebra of Loewy length $L+1$, i.e., $\la = \latrunc$}. In particular, the irreducible components of $\modlad$ are among the $\overline{\laySS}$, where $\SS$ traces the $\bd$-dimensional realizable semisimple sequences. The upcoming theory characterizes these  components in terms of their generic radical layerings $\SS$ (or, equivalently, in terms of their {\it generic modules\/} in the sense of Section \ref{sec5} below).  As in the special cases already mastered --  the local case and that of an acyclic quiver $Q$ --  the classification may be implemented on a computer; see Section 5.B.  However, the general algorithm is considerably more labor-intensive than the $\Theta$-test which applies to the local and acyclic cases.

As we will recall in Section \ref{sec5}, the generic properties of the modules in any component $\overline{\laySS}$ may be accessed via a single generic module $G(\SS)$.  A key asset of the truncated situation lies in the fact that such a module $G(\SS)$ is available on sight from $\SS$; detail will follow in Section 5.A below.  

Moreover, it is particularly easy to recognize realizability of semisimple sequences over truncated path algebras.  We recall the following from \cite[Criterion 3.2]{irredcompI}:

\begin{realcrit}  \label{crit4.1} Let $\bB = (\bB_{ij})$ be the adjacency matrix of $Q$, i.e., $\bB_{ij}$ is the number of arrows from $e_i$ to $e_j$. Then $\SS= (\SS_0,\dots,\SS_L)$ is realizable if and only if $\udim \SS_l \le (\udim\SS_{l-1}) \cdot \bB$ for all $1 \le l \le L$; the latter, in turn, is equivalent to realizability of the two-term sequences $(\SS_l,\SS_{l+1})$ in $(\la/J^2)\text{-mod}$ for $l < L$.
\qed
\end{realcrit}

In more intuitive terms: $\SS$ is realizable if and only if there exists an abstract skeleton with layering $\SS$.  Note moreover that, in the positive case, any such skeleton belongs to the {\it generic\/} set of  skeleta of the modules in $\overline{\laySS}$.    

Next, we find that the description of $\scaledrep ({\ge} \SS)$ may be simplified in the truncated situation, in that requirement (iii) of Lemma \ref{lem3.2} is now void. 

\begin{observation}  \label{obs4.2} \textbf{$\scaledrep ({\ge} \SS)$ is an affine space.}  Referring to the decomposition of $K^d$ induced by $\SS$ in Definition \ref{def3.1}, we obtain:  $\scaledrep ({\ge} \SS)$ consists of those points $f = (f_\alpha)_\alpha \in \bigl(\End_\la(K^d)\bigr)^ {Q_1}$ which satisfy the following conditions:  For any arrow $\alpha$ from $e_i$ to $e_j$:
\begin{enumerate}
 \item[$\bullet$] $f_{\alpha} (\K_{(l, r)}) = 0$ \ for all $r \ne i$, \ \  and
\item[$\bullet$] $f_{\alpha} (\K_{(l, i)}) \subseteq \bigoplus_{l+1 \le m \le L} \K_{(m, j)}$. 
\end{enumerate} 

 In particular, $\scaledrep({\ge} \SS)$ is a full affine space in this situation.  Indeed, the image of the closed immersion $\scaledrep ({\ge} \SS) \hookrightarrow \modlad$, which we presented in \ref{def3.3}, consists of {\it all\/} sequences of $d_i \times d_i$ matrices of the described lower triangular format.  Consequently, $\Filt \SS$, being a morphic image of $\GL(\bd) \times \scaledrep(\ge\SS)$, is irreducible as well.
\end{observation} 

This observation, in turn, allows us to derive a full characterization of the modules in $\overline{\laySS}$ from Theorem \ref{thm3.8}. 

\begin{theorem}  \label{thm4.3}  Suppose $\la$ is a truncated path algebra and $\SS$ a realizable semisimple sequence.  Then 
$$\overline{\laySS} = \Filt \SS.$$
In other words, a module $M$ belongs to $\overline{\laySS}$ precisely when $M$ has a filtration governed by $\SS$.

Dually, $\overline{\Corep \SS^*} = \Cofilt \SS^*$, where $\SS^*$ is the generic socle layering of the modules in $\laySS$. If $\overline{\laySS}$ is an irreducible component of $\modlad$, then
$$\overline{\Corep \SS^*} = \Cofilt \SS^* = \Filt \SS = \overline{\laySS}.$$ 
\end{theorem}

\begin{proof} Concerning the first equality: In light of Observation \ref{obs4.2}, the variety $\scaledrep(\ge\SS)$ is irreducible.  Therefore the open subset $\scaledrep \SS$ is dense in $\scaledrep(\ge\SS)$, meaning that the closure $\overline{\scaledrep \SS}$ in $\modlad$ contains $\scaledrep(\ge\SS)$.  Moreover, $\scaledrep \SS \subseteq \scaledrep(\ge \SS)$ by construction, whence we obtain $\scaledrep(\ge\SS) \subseteq \overline{ \scaledrep \SS } \subseteq \overline{ \laySS }$. Given that $\overline{ \laySS }$ is $\GL(\bd)$-stable, it follows that $\Filt \SS \subseteq \overline{ \laySS }$ due to 3.6. The reverse inclusion was established in Theorem \ref{thm3.8}.  The second assertion follows by duality (see \ref{thm3.12} and \cite[Corollary 3.4.b]{irredcompII}).  

In particular, duality guarantees that the varieties $\overline{\Corep \SS^*}$ are again irreducible.
For arbitrary $\SS$, we moreover find  $\laySS \subseteq \overline{\Corep \SS^*}$, since the modules in a dense open subset of $\laySS$ have socle layering $\SS^*$. In case $\overline{\laySS}$ is an irreducible component of $\modlad$, we thus infer $\overline{\laySS} = 
\overline{\Corep \SS^*}$, which completes the argument.
\end{proof}

The following consequence, addressing the relative sizes of the closures $\overline{\laySS}$, is now immediate.  It was independently obtained by I. Shipman with different methods; he also developed an algorithm for checking the considered inclusion via matrices of dimension vectors \cite{Shi}.  Algorithmic counterparts to the upcoming Corollary \ref{cor4.4} and Theorem \ref{thm4.5} will be addressed in 5.B.

\begin{corollary}  \label{cor4.4}  {\rm\textbf{Comparing the varieties $\overline{\laySS}$.}}  Let $\la$ be a truncated path algebra.  Moreover, suppose that $\SS$ and $\SS'$ are realizable semisimple sequences with the same dimension vector.  Then $\overline{\laySS}\subseteq \overline{\Rep \SS'}$ if and only if {\rm{(}}generically{\rm{)}} the modules in $\overline{\laySS}$ have filtrations governed by $\SS'$. \qed
\end{corollary} 

The upper semicontinuous map $ \Gabull: \modlad \rightarrow \NN$ of {\rm \ref{cor3.11}}
detects all irreducible components of $\modlad$.  Indeed, $\SS$ is the generic radical layering of an irreducible component of $\modlad$ if and only if $\Gabull$ attains the value $1$ on $\laySS$.
We record this as follows. 
    
\begin{theorem}  \label{thm4.5} Let $\la$ be a truncated path algebra.  If $\, \SS^{(1)}, \dots, \SS^{(m)}$ are the distinct $\bd$-dimensional semisimple sequences $\SS$ with $1 \in \Gabull(\laySS)$, then 
$$\Filt(\SS^{(1)}) = \overline{\Rep \SS^{(1)}}, \; \dots\; ,\; \Filt(\SS^{(m)}) = \overline{\Rep \SS^{(m)}}$$ 
are the distinct irreducible components of $\modlad$.
\end{theorem}

\begin{proof}  Suppose $\SS$ is a realizable $\bd$-dimensional semisimple sequence. If $1 \in \Gabull(\laySS)$, then $\laySS \nsubseteq \Filt(\SS') = \overline{\laySS'}$ for any semisimple sequence $\SS' \ne \SS$, whence $\overline{\laySS}$ is an irreducible component of $\modlad$.

If, on the other hand, $1 \notin \Gabull(\laySS)$, then every module in $\laySS$ is contained in some variety $\Filt \SS'$, where $\SS'$ is a realizable semisimple sequence different from $\SS$. Therefore, 
$$\overline{\laySS}\ \  \subseteq\  \bigcup_{\SS' \ \text{realizable},\ \SS' \ne \SS} \Filt \SS' \ \ = \ \bigcup_{\SS' \ \text{realizable},\ \SS' \ne \SS} \overline{\laySS'},$$
the final equality being part of \ref{thm4.3}. 
Irreducibility of $\overline{\laySS}$ thus implies  $\overline{\laySS} \subseteq \overline{\laySS'}$ for some $\SS' \ne \SS$, which shows that $\overline{\laySS}$ fails to be maximal irreducible. \end{proof}

%%%%%%%%%%%%%%%%%%%%%%%%%%%%%
\section{Applications of Section \ref{sec4}:  Generic modules for\\ the components over truncated path algebras}
\label{sec5}

Barring Example \ref{ex5.2}(b), {\it $\la$ will, throughout this section, stand for a truncated path algebra of Loewy length $L+1$.  Moreover, $\bd$ will be a dimension vector of  $\la$\/}. 

 If one extends the base field $K$ of $\la$ to an algebraically closed field of infinite transcendence degree over its prime field $K_0$, neither the description of the components of $\modlad$, nor the generic properties of their modules will be affected; see \cite[Section 2.B]{irredcompII}. This means that, in developing a generic representation theory for the irreducible components of $\modlad$, one does not lose generality in assuming that $\trdeg(K:K_0) = \infty$. 
%%%%%%%%%%%%%
\subsection*{5.A.~Generic modules}  \hfill\par
\smallskip

Assume  that $K$ has infinite transcendence degree over $K_0$, and let $\SS$ be a realizable $\bd$-dimensional semisimple sequence. Given that $\la = \latrunc$, we will denote the coordinatized projective $\latrunc$-projective cover $\Ptrunc = \bigoplus_{1 \le r \le t} \la \bz_r$ of $\SS_0$ (cf.~Section \ref{sec2}) more simply by $P$.  

Let $\S$ be any skeleton with layering $\SS$.  Then the following module $G = G(\SS)$ is generic for $\overline{\laySS}$ in the strict sense of \cite[Definition 4.2]{BHT}:  
$$G = P/ C, \ \  \text{where} \ \  C =  \sum_{\bq\; \S\text{-critical}} \la \biggl(  \bq - \sum_{\bp \in \S_\bq} c_{\bq, \bp} \bp \biggr)$$
for some family $(c_{\bq, \bp})_{\bq\; \S\text{-critical},\, \bp \in \S_\bq}$ of scalars which is algebraically independent over $K_0$.  
That $G$ is {\it generic\/} means that $G$ has all those generic properties of the modules in $\overline{\laySS}$ which are invariant under Morita self-equivalences $\lamod \rightarrow \lamod$ induced by automorphisms of $K$ over $\overline{K_0}$.  Moreover, $G$ is unique relative to this property, up to such a Morita self-equivalence. We refer to \cite[Theorem 5.12]{BHT}, and to \cite[Section 4]{BHT} for a more general statement addressing arbitrary path algebras modulo relations.
\smallskip

\textbf{Filtrations of generic modules:} In particular, the preceding comments ensure that tests for semisimple sequences which generically govern filtrations of the modules in $\overline{\laySS}$ may be confined to ``the" generic module $G = G(\SS)$.
\smallskip

Caveat: Suppose $G$ is a generic module for an irreducible component of $\modlad$.  While the combination of \ref{cor3.11} and \ref{thm4.5} guarantees that the radical layering $\SS(G)$ is the only {\it realizable\/} semisimple sequence to govern a filtration of $G$, there will in general be further, non-realizable, sequences governing suitable filtrations.  For instance, let $Q$ be the quiver 
$\xymatrixrowsep{2pc}\xymatrixcolsep{1pc}
\xymatrix{
4  &1 \ar[l] \ar[r]^{\alpha}  &2 \ar[r]  &3
}$ and $\la$ any truncated path algebra based on $Q$.
  If $\bd = (0, 1,1,1)$, then $\modlad$ is irreducible with generic module $G = \la \alpha \oplus S_4$ for any truncation $\la$ of $KQ$.  In particular, $\SS(G) = (S_2 \oplus S_4, S_3)$ is the only realizable semisimple sequence governing all modules with dimension vector $\bd$.  In case $\la$ has Loewy length $2$, The sequence $(S_2, S_3 \oplus S_4)$ also governs a filtration of $G$; if the Loewy length of $\la$ is $3$, then  $(S_2, S_3, S_4)$ and $(S_4, S_2, S_3)$ are additional (non-realizable) semisimple sequences governing filtrations of $G$. 
  
%%%%%%%%%%%%%

\subsection*{5.B.~Algorithmic aspect of Corollary \ref{cor4.4} and Theorem \ref{thm4.5}}  \hfill\par
\smallskip

5.A tells us that, for any two realizable $\bd$-dimensional semisimple sequences $\SS$ and $\SS'$, we have
$$\laySS \subseteq \overline{\laySS'}\ \ \iff \ \ G(\SS) \in \Filt \SS'.$$ 
From Lemma \ref{lem3.6}, we moreover know that $\Filt \SS'$ is the $\GL(\bd)$-stable hull of $\scaledrep ({\ge} \SS')$.   Hence, if the point $(G_\alpha)_{\alpha \in Q_1} \in \laySS$ represents the isomorphism class of $G(\SS)$, the question of whether $G(\SS) \in \Filt \SS'$ boils down to the question of whether the matrices $G_\alpha$ are ``simultaneously" similar $($i.e., similar by way of a single element of $\GL(\bd)$) to matrices having the lower triangular format $F_\alpha$ characterizing the points in $\scaledrep(\ge \SS')$.   This format is spelled out in \ref{def3.3}.

Given that there are only finitely many $\bd$-dimensional semisimple sequences to be compared, this means in particular that the decision of whether or not $\overline{\laySS}$ is a component of $\modlad$ is algorithmic.

%%%%%%%%%%%%%%
\subsection*{5.C.~Interconnections among the components}  \hfill\par
\smallskip

The following statement rephrases a result of Crawley-Boevey and Schr\"oer \cite[Theorem 1.1]{CBS} in terms of generic modules: If $G$ is a generic module for an irreducible component $\C$ of $\modlad$ and $G = \bigoplus_{1 \le j \le s} G_j$ is a decomposition into direct summands, then each $G_j$ is generic for an irreducible component of $\Rep_{\udim G_j} (\la)$.  
Over a truncated path algebra, this result may be sharpened as follows.  

Call a submodule $M$ of $N$ {\it layer-stably embedded\/} in $N$ if $J^lM = M \cap J^l N$ for all $l \le L$.  As a  consequence of Theorem \ref{thm4.5}, we obtain:

\begin{theorem}  \label{thm5.1} Suppose that $\la$ is a truncated path algebra and $\overline{\laySS}$ an irreducible component of $\modlad$ with generic module $G$.  If $G' \subseteq G$ is a layer-stably embedded submodule of $G$ with $\SS(G') = \SS'$ and $\udim G' = \bd'$, then $\overline{\Rep \SS'}$ is an irreducible component of $\Rep_{\bd'} (\la)$ with generic module $G'$. 
\end{theorem}

\begin{proof}  Let $H := G'$ be layer-stably embedded in $G$.  From \cite[Corollary 3.2]{irredcompII} we know that $H$ is generic for $\Rep \SS' = \Rep \SS(H)$.  Thus only the status of $\overline{\Rep \SS'}$ as a potential component of $\Rep_{\bd'} (\la)$ needs to be addressed.   

Assume that $\Rep \SS'$ fails to be an irreducible component of $\Rep_{\bd'} (\la)$.  In view of Theorem \ref{thm4.5}, this means that $H$ has a filtration governed by some realizable semisimple sequence $\SS''$ which is strictly smaller than $\SS'$, say $H = H_0 \supseteq H_1 \supseteq \cdots \supseteq H_L \supseteq H_{L+1} = 0$; by definition, $\SS''_l = H_l/ H_{l+1}$.   We aim at constructing a submodule filtration $G = G_0 \supseteq \cdots \supseteq G_L \supseteq 0$ which, in turn, is governed by a realizable semisimple sequence $\SShat$ strictly smaller than $\SS$.  Another application of Theorem \ref{thm4.5} will then show that $\overline{\laySS}$ is not an irreducible component of $\modlad$, contrary to our hypothesis.

For $l \le L$, let $\pi_l: G \rightarrow G/ J^{l+1} G$ denote the quotient map.  We recursively choose submodules $U_l$ of $J^l G$ such that 
\begin{equation}  \label{5.1.1}
J^{l+1}G \subseteq U_l \subseteq J^lG, \quad U_l \subseteq JU_{l-1}, \quad \text{and} \quad 
J^l G/ J^{l+1} G = \pi_l(J^l H) \oplus \pi_l (U_l).
\end{equation}
First, semisimplicity of $G/JG$ implies that $G/JG = \pi_0(H) \oplus \pi_0(U_0)$ for some $U_0 \subseteq G$, and since $U_0$ may be replaced by $U_0 + JG$, there is no loss of generality in assuming that $JG \subseteq U_0$. If $U_0,\dots,U_k$ for some $k<L$ have been chosen so as to satisfy \eqref{5.1.1}, we have $J^kG = J^kH + U_k$ by Nakayama's Lemma, whence $J^{k+1}G = J^{k+1}H + JU_k$. Consequently, $J^{k+1}G/J^{k+2}G = \pi_{k+1}(J^{k+1}H) \oplus \pi_{k+1}(U_{k+1})$ for some $U_{k+1} \subseteq JU_k$. On replacing $U_{k+1}$ by $U_{k+1}+J^{k+2}G$, we obtain \eqref{5.1.1} for $l=k+1$. Finally, set $U_{L+1} := 0$.

Now define $G_l := H_l + U_l$ for $l \le L+1$. That the consecutive factors of the sequence 
\begin{equation}  \label{5.1.2}
G = G_0 \supseteq G_1 \supseteq \cdots \supseteq G_L \supseteq G_{L+1} = 0
\end{equation} 
are semisimple, i.e., $JG_l \subseteq G_{l+1}$ for $l \le L$, is straightforward from our construction. Indeed, $JH_l \subseteq H_{l+1}$ and $JU_l \subseteq J^{l+1}G = J^{l+1}H + U_{l+1} \subseteq H_{l+1} + U_{l+1}$. Let $\SShat$ be the semisimple sequence governing the filtration \eqref{5.1.2}. Remark \ref{rem3.7}(2) tells us that $\SShat \le \SS$. 

Suppose $m$ is minimal with the property that $J^m H  \subsetneqq H_m$. Such an index $m$ exists, since $\SS'' < \SS'$. Then $m \ge 1$. Using layer-stability of $H$ in $G$, we derive $J^m G = J^m H + U_m\subsetneqq H_m + U_m = G_m$. On the other hand, $G_l = J^lG$ for $l<m$, so
 that the first discrepancy between the downward filtration \eqref{5.1.2} and the radical filtration of $G$ occurs at $l = m$.  More specifically, $\udim \SShat_{m-1} = \udim (G_{m-1} / G_m) <  \udim (G_{m-1}/ J^m G) =  \udim J^{m-1}G / J^m G = \udim \SS_{m-1}$.  This yields $\SShat <  \SS$. 

It remains to be verified that $\SShat$ is realizable. To do so, we make repeated use of Criterion \ref{crit4.1}. Again, $\bB$ is the adjacency matrix of $Q$. First we note that realizability of $\SS$ and $\SS''$  entails
\begin{equation}  \label{5.1.3}
\udim J^lG/J^{l+1}G \le (\udim J^{l-1}G/J^lG)\cdot \bB \quad\text{and}\quad \udim H_l/H_{l+1} \le (\udim H_{l-1}/H_l)\cdot \bB
\end{equation}
for $1\le l\le L$. Therefore $\,\udim G_l/G_{l+1} \le (\udim G_{l-1}/G_l) \cdot \bB\,$ for $1\le l\le m-2$. 

Invoking \eqref{5.1.1}, we find that, for $1\le l\le L$,
\smallskip

\centerline{$G_l/J^{l+1}G = (H_l+J^{l+1}G)/J^{l+1}G \oplus (U_l/J^{l+1}G)\ \ \text{and}$}
\smallskip

\centerline{$  G_{l+1}/J^{l+1}G = (H_{l+1}+ J^{l+1}G)/J^{l+1}G$,}
\smallskip

\noindent where the sum in the first equation is direct because $H_l \cap U_l \subseteq H\cap J^lG = J^lH$ implies $H_l \cap U_l = J^lH \cap U_l \subseteq J^{l+1}G$.
We also have $(H_l+J^{l+1}G)/(H_{l+1}+ J^{l+1}G) \cong H_l/H_{l+1}$, since layer-stability of $H$ in $G$ guarantees that $H_l\cap J^{l+1}G \subseteq J^{l+1}H \subseteq H_{l+1}$. Consequently,
\begin{equation}  \label{5.1.5}
G_l/G_{l+1} \cong (H_l/H_{l+1}) \oplus (U_l/J^{l+1}G), \quad\text{for} \;\; 1\le l\le L.
\end{equation}

Since $U_l \subseteq JU_{l-1}$, we moreover obtain
\begin{equation}  \label{5.1.6}
\udim U_l/J^{l+1}G \le \udim JU_{l-1}/J(J^lG) \le (\udim U_{l-1}/J^lG) \cdot \bB, \quad\text{for} \;\; 1\le l\le L.
\end{equation}
Combining \eqref{5.1.6} with \eqref{5.1.3} and \eqref{5.1.5} yields $\udim G_l/G_{l+1} \le (\udim G_{l-1}/G_l) \cdot \bB$ for $1\le l\le L$, which shows that $\SShat$ is realizable as required. 
\end{proof}

The following examples demonstrate: {\bf (a)} that the conclusion of \ref{thm5.1} does not extend to arbitrary top-stably embedded submodules $G'$ of $G$, i.e., to submodules $G'$ satisfying only $J G' = G' \cap JG$, and {\bf (b)} that \ref{thm5.1} has no analogue for nontruncated $\la$ in general.

\begin{examples}  \label{ex5.2}  \textbf{Demonstrating the sharpness of \ref{thm5.1}.} Consider the quivers
$$\xymatrixrowsep{1.5pc}\xymatrixcolsep{1.75pc}
\xymatrix{
Q_1 :  &1 \ar[r] \ar@/^1pc/[rr] &2 \ar[r] &3  &4 \ar@/^1pc/[ll] &5 \ar[l] &&Q_2 : &4 \ar@/^1pc/[rr]^{\delta} &1 \ar[r] \ar@/_1pc/[rr] &2 \ar[r]^{\beta} &3
}$$
\smallskip

\textbf{(a)}  Let $\la$ be the truncated path algebra of Loewy length $3$ based on the quiver $Q_1$. For $\bd = (1, 1, 1, 1, 1)$, the variety $\modlad$ has two irreducible components, with generic radical layerings $\SS^{(1)} := (S_1 \oplus S_5, S_3 \oplus S_4, S_2)$ and $\SS^{(2)} := (S_1 \oplus S_5, S_2 \oplus S_4, S_3)$ and generic modules $G_1$ and $G_2$ as graphed below.
$$\xymatrixrowsep{1pc}\xymatrixcolsep{0.7pc}
\xymatrix{
&1 \edge[d] \dashedge[ddrr] &&5 \edge[d] &&&&&1 \edge[d] \dashedge@/^1pc/[dd] &&5 \edge[d]  \\
G_1 :  &3 &&4 \edge[d]  &&&&G_2 :  &2 \edge[d] &{\;\;\;\textstyle\bigoplus} &4  \\
&&&2  &&&&&3
}$$
Clearly, the top-stably embedded submodule $G'$ of $G_1$ generated by any element $z = e_1 z \in G_1$ has dimension vector $\bd' := (1,1,1,0,0)$.  On the other hand, the sequence $ \SS(G') = (S_1, S_2 \oplus S_3, 0)$ fails to be the generic radical layering of an irreducible component of $\Rep_{\bd'}(\la)$, the latter variety being irreducible with uniserial generic modules.
\smallskip

\textbf{(b)}  Now let $\la = KQ_2 / \langle \beta \delta \rangle$ and $\bd := (1,1,1,1)$.  Then, again, $\modlad$ consists of two irreducible components. Their generic modules are graphed below:
$$\xymatrixrowsep{1pc}\xymatrixcolsep{0.7pc}
\xymatrix{
&&1 \edge[dl] \edge[dr] &&4 \edge[dl]  &&&&&1 \edge[d] \dashedge@/^1pc/[dd] &&4 \drbl  \\
G_1 :  &3 &&2  &&&&&G_2 :  &2 \edge[d] &{\;\;\;\textstyle\bigoplus}  \\
&&&& &&&&&3
}$$
The submodule $G'$ of $G_1$ generated by any element $z = e_1 z \in G_1$ has dimension vector $\bd' := (1,1,1,0)$ and is layer-stably embedded in $G_1$ this time.  Nonetheless, $\overline{\Rep \SS(G')}$ fails to be an irreducible component of $\Rep_{\bd'} (\la)$.  Indeed, once again, $\Rep_{\bd'} (\la)$ is irreducible and its generic modules are uniserial.
\qed
\end{examples}

%%%%%%%%%%%%%%%%
\section{Examples illustrating the theory.  The interplay $\modlad \longleftrightarrow \Rep_\bd(\latrunc)$}
\label{sec6}

\subsection*{6.A.~Illustrations of the truncated case}  \hfill\par
\smallskip

{\it In this subsection, $\la$ denotes a truncated path algebra. \/}

In sifting the radical layerings of the components of $\modlad$ out of the set $\Seq(\bd)$, it is computationally advantageous to supplement $\Gabull$ by the map $\Theta$ of equation (1.1) in Section \ref{intro}, or by the upgraded map $\Theta^+$ to be introduced next.    
  
Example 4.8 in \cite{irredcompI} shows that $\Theta$ fails to detect all irreducible components in the general truncated case.  However, in that instance (as in many others), supplementing $\Theta$ by path ranks compensates for the blind spots of $\Theta$.  Here the {\it path rank\/} of a finite dimensional $\la$-module $M$ is the tuple $( \dim\, pM)_p \in \ZZ^{\tau}$, where $\tau$ is the set of paths in $KQ \setminus I$. 
 Set $f(M)  = (- \dim\, pM)_p$, and let $f^*(M)$ be the negative of the path rank of the right $\la$-module $D(M)$.  Clearly, the map
$$\Theta^+: \modlad \rightarrow \Seq(\bd) \times \Seq(\bd) \times \ZZ^\tau \times \ZZ^\tau,  \ \ x \mapsto \bigl(\SS(M_x), \SS^*(M_x), f(M_x), f^*(M_x) \bigr),$$
is in turn upper semicontinuous.  Therefore, it is generically constant on the varieties $\laySS$.  In particular, those closures $\overline{\laySS}$ on which $\Theta^{+}$ attains its minimal values (relative to the componentwise partial order on the codomain) are components of $\modlad$.   Yet, part (c) of the next example attests to the fact that the augmented upper semicontinuous map $\Theta^+$ still leaves certain components undetected in general.  We use $\Gabull$ to fill in what $\Theta^+$ fails to pick up.  

\begin{example}  \label{ex6.1} Let $\la$ be the truncated path algebra of Loewy length $4$ based on the quiver $Q$ of  Example \ref{ex2.4}, and take $\bd = (2,2)$.  The semisimple sequences which are in the running as potential generic radical layerings of components of $\modlad$ are:
$$\SS^{(1)} = (S_1, S_2, S_1, S_2), \, \SS^{(2)} = (S_2, S_1, S_2, S_1), \, \SS^{(3)} =  (S_1, S_2^2, S_1, 0),\, \SS^{(4)} = (S_2, S_1^2, S_2, 0),$$ 
$$\SS^{(5)} = (S_1^2, S_2^2, 0, 0), \  \SS^{(6)} = (S_2^2, S_1^2, 0, 0), \ \SS^{(7)} = (S_1 \oplus S_2, S_1 \oplus S_2, 0, 0),$$
$$\SS^{(8)} = (S_1 \oplus S_2, S_1, S_2, 0), \ \ \text{and} \ \ \SS^{(9)} = (S_1 \oplus S_2, S_2, S_1, 0).$$
The list excludes the sequences which are not realizable for any choice of $r$ and $s$, such as  $(S_1, S_1 \oplus S_2, S_2, 0)$ and $(S_1, S_2, S_1 \oplus S_2, 0)$, as well as the radical layering $\SS^{(0)}$ of the semisimple module, given that $\laySS^{(0)}$ is contained in all nonempty varieties $\overline{\laySS}$.  Except for $\SS^{(3)}$ and $\SS^{(4)}$, all sequences on the list are realizable for arbitrary positive integers $r,s$. 

Theorem \ref{thm4.5} allows us to discard $\SS^{(j)}$ for $j = 7, 8, 9$ from the list of possible generic radical layerings of irreducible components:  Indeed, the modules in $\Rep \SS^{(7)}$ are generically decomposable, which makes it evident that they have filtrations governed by both $\SS^{(1)}$ and $\SS^{(2)}$.  
Any generic module $G_8$ for $\Rep \SS^{(8)}$ has hypergraph 
$$\xymatrixrowsep{2pc}\xymatrixcolsep{1.5pc}
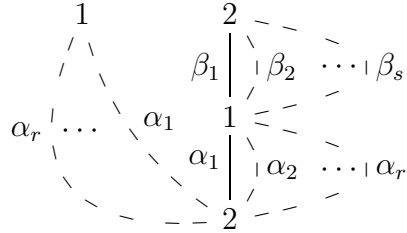
\xymatrix{
1 \dashedge@/_2ex/[ddrr]^{\alpha_1} \dashedge@/_8ex/[ddrr]_(0.3){\alpha_r}^(0.3){\cdots}  &&2 \edge[d]_{\beta_1} \dashedge@/^2ex/[d]^{\beta_2} \dashedge@/^10ex/[d]^{\beta_s}_{\cdots}  \\
&&1  \edge[d]_(0.4){\alpha_1} \dashedge@/^2ex/[d]^{\alpha_2} \dashedge@/^10ex/[d]^{\alpha_r}_{\cdots}  \\
&&2
}$$
\noindent Clearly, $G_8$ is generated by elements $z_1 = e_1 z_1$ and $z_2 = e_2 z_2$, and the following submodule chain is governed by $\SS^{(1)}$:  
$$G_8 \supseteq \la z_2 \supseteq \la \beta_1 z_2 \supseteq \la \alpha_1 \beta_1 z_2 \supseteq 0.$$ 
Consequently, $\Rep \SS^{(8)} \subseteq \Filt \SS^{(1)}$ by \ref{cor4.4}.  An analogous argument shows $\Rep \SS^{(9)} \subseteq \Filt \SS^{(2)}$. 

On the other hand, $\C_j : = \overline{\Rep \SS^{(j)}}$ for $j = 1,2$ are components of $\modlad$ for all choices of $r,s \ge 1$ by Theorem \ref{thm4.5}, since $\Gamma(U) = 1$ for any uniserial module $U$.  Hence only the sequences $\SS^{(j)}$ for $3 \le j \le 6$ require discussion by cases. We consider only the cases when $r \ge s$, due to the symmetry of the quiver $Q$.
\smallskip

\textbf{(a)} Let $r = s = 1$.  Then $\modlad$ has precisely two irreducible components, namely $\C_j =  \overline{\Rep \SS^{(j)} }$ for $j = 1, 2$.  We rule out the remaining sequences. First, $\SS^{(3)}$ and $\SS^{(4)}$ fail to be realizable when $r=s=1$. Generically, the modules in $\Rep \SS^{(5)}$ are direct sums of two uniserials with radical layering $(S_1,S_2,0,0)$, and such a module has a filtration governed by $\SS^{(1)}$.  Thus, $\Rep \SS^{(5)} \subseteq \Filt \SS^{(1)} = \C_1$. Similarly, $\Rep \SS^{(6)} \subseteq \C_2$.
\smallskip

\textbf{(b)} Let $r = 2$, $s = 1$.  Then $\modlad$ again has precisely two irreducible components,  $\C_1$ and $\C_2$.  Concerning $\SS^{(3)}$:  A generic module $G_3$ for $\Rep \SS^{(3)}$ has a hypergraph of the form 
$$\xymatrixrowsep{1.3pc}\xymatrixcolsep{1pc}
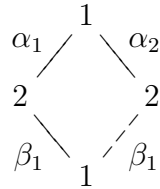
\xymatrix{
&1 \edge[dl]_{\alpha_1} \edge[dr]^{\alpha_2}  \\
2 \edge[dr]_{\beta_1} &&2 \dashedge[dl]^{\beta_1}  \\
&1
}$$
In particular, the socle of $G_3 = \la z$ contains a copy of $S_2$, namely $\la (\alpha_1 - k \alpha_2)z$ for a suitable scalar $k \in K^*$.  We deduce that the submodule chain 
$$G_3 \supseteq JG_3  \supseteq \la (\alpha_1 - k \alpha_2) z + \la \beta_1 \alpha_1 z \supseteq \la (\alpha_1 - k \alpha_2) z \supseteq 0$$
is governed by $\SS^{(1)}$, showing $\Rep \SS^{(3)} \subseteq \Filt \SS^{(1)} = \C_1$.  (On the side, we mention that $\Rep \SS^{(3)}$ is not contained in $\C_2$ because the sequences $\SS^{(2)}$ and $\SS^{(3)}$ are not comparable under the dominance order.) 

The sequence $\SS^{(4)}$ fails to be realizable for $s = 1$.
As for $\SS^{(5)}$:  Generically the modules in $\Rep \SS^{(5)}$ decompose in the form shown at the end of \ref{ex2.4}(a), whence $\Rep \SS^{(5)} \subseteq \C_1$.  (Clearly, $\Rep \SS^{(5)}  \not\subseteq \C_2$, because $\SS^{(5)}$ is not comparable to $\SS^{(2)}$.)  A routine check shows that $\Rep \SS^{(6)}$ is contained in $\C_2$, but not in $\C_1$.  
\smallskip

\textbf{(c)} Let $r \ge 3$, $s = 1$.  Then the variety $\modlad$ has three  irreducible components, namely $\C_j = \overline{\Rep \SS^{(j)} }$, for $j = 1, 2, 5$.  The status of $\C_1$, $\C_2$ being clear, we focus on the variety $\Rep \SS^{(5)}$ with generic module $G_5$ as depicted at the end of \ref{ex2.4}(b).  Again, we prove our claim regarding $\C_5$ via Theorem \ref{thm4.5}:  To see that $\SS^{(5)} = \SS(G_5)$ is the only realizable semisimple sequence governing a filtration of $G_5$, we note that the only other realizable sequence not ruled out by $\Theta$ (i.e., with a $\Theta$-value $< \Theta(G_5)$) is $\SS^{(1)}$.  To verify, without computational effort, that $\SS^{(1)}$ does not govern any filtration of $G_5$, it suffices to observe that, for any module $N$ in $\Filt \SS^{(1)}$, we have $S_1 \subseteq N/\la x$ for some $x \in e_2 N$. On the other hand, it is readily checked that $S_1 \nsubseteq G_5/ \la x$ for all elements $x  \in e_2 G_5$, which shows $\Gamma(G_5) = 1$ as required.  To link up with the remarks preceding \ref{ex6.1} finally, we point out that $\Theta^+ (G_1) < \Theta^+(G_5)$, whence the $\Theta^+$-test fails to detect the status of $\overline{\Rep \SS^{(5)} }$ as an irreducible component of $\modlad$.
 
To see that $\SS^{(j)}$ for $j = 3, 4, 6$ do not arise as generic radical layerings of irreducible components of $\modlad$, one  may follow the patterns of part (b).
\smallskip

\textbf{(d)} Moving to $r \ge 3$ and $s = 2$ raises the number of irreducible components of $\modlad$ to $5$.  We first show that $\overline{\Rep \SS^{(3)}}$ is now a component.  Generically, the modules  in $\Rep \SS^{(3)}$ have hypergraph 
$$\xymatrixrowsep{1.8pc}\xymatrixcolsep{1pc}
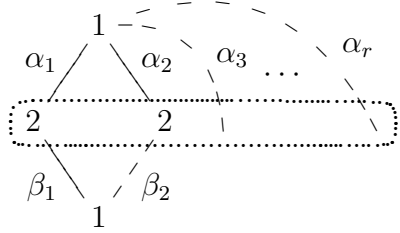
\xymatrix{
&1 \edge[dl]_{\alpha_1} \edge[dr]^{\alpha_2}  &&&\ar@{}[d]|{\;\;\cdots}  \\
2 \horizpool{6} \edge[dr]_{\beta_1}  &&2 \dashedge[dl]^{\beta_2}  & \save+<0ex,-0.1ex> \dashedge@/_3ex/[ull]_(0.3){\alpha_3} \restore &&& \save+<0ex,-0.1ex> \dashedge@/_5ex/[ulllll]_(0.2){\alpha_r} \restore  \\
&1
}$$
Again, the only $\overline{\Rep \SS^{(j)}}$ (for $j\le 6$) potentially containing $\Rep \SS^{(3)}$ is $\overline{\Rep \SS^{(1)}} = \Filt \SS^{(1)}$.  Since the modules in $\Filt \SS^{(1)}$ clearly contain a copy of $S_2$ in their socle, while $G_3$ does not, this possibility is ruled out, and our claim is justified.

The discussion of $\overline{\Rep \SS^{(4)}}$ is analogous, in that  the only $\overline{\Rep \SS^{(j)}}$ (for $j\le 6$) potentially containing $\Rep \SS^{(4)}$ is $\overline{\Rep \SS^{(2)}} = \Filt \SS^{(2)}$, and the modules in $\Filt \SS^{(2)}$ contain a copy of $S_1$ in their socle, while a generic module for $\Rep \SS^{(4)}$ does not.

As in part (c), one shows that $\overline{\Rep \SS^{(5)}}$ is a component of $\modlad$. 
On the other hand, $\overline{\Rep \SS^{(6)}}$ still fails to be a component;  the argument used in part (b) (in that case, to exclude $\overline{\Rep \SS^{(5)}}$ from the list of components for $r = 2$) may now be applied to $s = 2$.
\smallskip

\textbf{(e)}  Finally, let $r \ge 3$ and $s \ge 3$.  Then all of the varieties $\overline{\Rep \SS^{(j)}}$ for $j = 1, \dots, 6$ are irreducible components of $\modlad$.  The argument backing the status of $\SS^{(6)}$ follows the reasoning we used to confirm $\overline{\Rep \SS^{(5)} }$ as a component of $\modlad$ in part (c).  For $r = s = 3$, hypergraphs of generic modules for the components $\overline{\Rep \SS^{(j)} }$, $j = 1,3,5$, are shown below.  Due to symmetry, the generic structure of the modules in the remaining components is obtained by swapping the roles played by the vertices $1$ and $2$.
$$\xymatrixrowsep{1.8pc}\xymatrixcolsep{1.25pc}
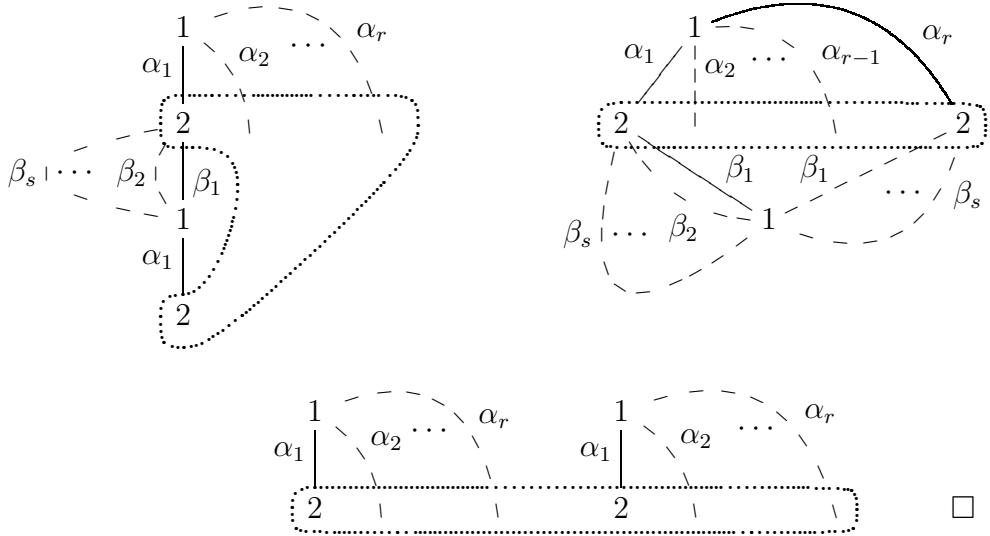
\xymatrix{
1 \edge[d]_(0.4){\alpha_1} &&&  &&&&&1 \edge[dl]_(0.4){\alpha_1} \edge@/^5ex/[drrrr]^(0.75){\alpha_r}  \\
2 \edge[d]^(0.6){\beta_1} \dashedge@/_2ex/[d]_{\beta_2} \dashedge@/_10ex/[d]_{\beta_s}^{\cdots} & \save+<0ex,-0.1ex> \dashedge@/_/[ul]_{\alpha_2} \restore  && \save+<0ex,-0.1ex> \dashedge@/_5ex/[ulll]_(0.3){\alpha_r}^(0.4){\cdots} \restore  &&&&2 \horizpool{5} \edge[drr]^(0.7){\beta_1} \dashedge@/_2ex/[drr]_(0.6){\beta_2} \dashedge@/_10ex/[drr]_(0.25){\beta_s}^(0.25){\cdots} & \save+<0ex,-0.1ex> \dashedge[u]_(0.6){\alpha_2} \restore  && \save+<0ex,-0.1ex> \dashedge@/_3ex/[ull]_(0.3){\alpha_{r-1}}^(0.5){\cdots} \restore  &&2 \dashedge[dlll]_(0.7){\beta_1}  \\
1 \edge[d]_(0.4){\alpha_1}   &&&  &&&&&&1 \dashedge@/_5ex/[urrr]_(0.8){\beta_s}^(0.7){\cdots}  \\
2 \save[0,0]+(0,3);[0,0]+(0,3) **\crv{~*=<2.5pt>{.} [0,0]+(4,3) &[-2,1]+(2,-3) &[-2,0]+(0,-3) &[-2,0]+(-3,-3) &[-2,0]+(-3,0) &[-2,0]+(-3,4) &[-2,0]+(0,4) &[-2,3]+(0,4) &[-2,3]+(5,4) &[-2,3]+(5,0) &[-2,3]+(5,-4) &[0,0]+(4,-4)  &[0,0]+(0,-4) &[0,0]+(-3,-4) &[0,0]+(-3,0) &[0,0]+(-3,3)}\restore &&&  \\  
&&1 \edge[d]_(0.4){\alpha_1}  &&&&&1 \edge[d]_(0.4){\alpha_1}  \\ 
&&2 \horizpool{8} & \save+<0ex,-0.1ex> \dashedge@/_/[ul]_{\alpha_2} \restore && \save+<0ex,-0.1ex> \dashedge@/_5ex/[ulll]_(0.3){\alpha_r}^(0.4){\cdots} \restore  &&2 &\save+<0ex,-0.1ex> \dashedge@/_/[ul]_{\alpha_2} \restore && \save+<0ex,-0.1ex> \dashedge@/_5ex/[ulll]_(0.3){\alpha_r}^(0.4){\cdots} \restore  &&\square
}$$
\end{example}
\medskip
 
 \subsection*{Consequences of the ``truncated" theory, exemplified by \ref{ex6.1}}  \hfill\par
\smallskip

\textbf{(1) Allocation of modules to the components.} Once the irreducible components $\overline{\laySS^{(j)}}$ of $\modlad$ have been pinned down, by way of Theorem \ref{thm4.5} say, one is in a position to list the components containing any given $\bd$-dimensional $\la$-module $M$.  Indeed, compiling this list amounts to deciding which of the $\SS^{(j)}$ govern filtrations of $M$; as was pointed out in Section 5.B, there is an algorithm for carrying out this task.  

In Example \ref{ex6.1} with $r = 3$ and $s \ge 1$, for instance, any module $M$ with hypergraph 
$$\xymatrixrowsep{2pc}\xymatrixcolsep{1.5pc}
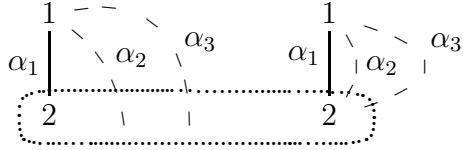
\xymatrix{
1 \edge[d]_{\alpha_1}  &&&&1  \edge[d]_(0.4){\alpha_1} \dashedge@/^2ex/[d]^{\alpha_2} \dashedge@/^7ex/[d]^(0.4){\alpha_3}    \\
2 \save[0,0]+(0,3.5);[0,0]+(0,3.5) **\crv{~*=<2.5pt>{.} [0,4]+(0,3.5) &[0,4]+(6,3.5) &[0,4]+(6,0) &[0,4]+(6,-3.5) &[0,4]+(0,-3.5) &[0,0]+(0,-3.5) &[0,0]+(-4,-3.5) &[0,0]+(-4,0) &[0,0]+(-4,3.5)}\restore & \save+<0ex,-0.1ex> \dashedge@/_/[ul]_(0.4){\alpha_2} \restore & \save+<0ex,-0.1ex> \dashedge@/_4ex/[ull]_(0.3){\alpha_3} \restore &&2
}$$
\smallskip

 \noindent belongs to the components $\C_1 = {\Filt \SS^{(1)}}$ and $\C_5 = {\Filt \SS^{(5)}}$, but does not have a filtration governed by $\SS^{(j)}$ for $j\in \{2, 3, 4 ,6\}$.  Therefore, $M$ belongs to precisely two of the irreducible components of $\modlad$, namely to $\C_1$ and $\C_5$.   
 \smallskip
 
\textbf{(2) Comparing the generic behavior of the finite dimensional $\la$-modules to that of the finite dimensional $KQ$-modules.}  Examples \ref{ex6.1}(a -- e) place a spotlight on the fact that, in the presence of oriented cycles, the generic representation theory of the path algebra $KQ$ may be ``disjoint" from that of its truncations in the following sense:  For $r,s \ge 1$, we have $J(KQ) = 0$, and for $\bd = (2,2)$ the modules in the irreducible variety $\Rep_\bd (KQ)$ are generically simple.  Since generically the latter modules are not annihilated by any path in $KQ$, we find the variety $\Rep_\bd\bigl(KQ/ \langle \text{the paths of length}\ 4 \rangle \bigr)$ to be contained in the boundary of a dense open subset of $\Rep_\bd(KQ)$.

\subsection*{6.B.~Information on the components of $\modlad$ from those of $\Rep_\bd(\latrunc)$}  \hfill\par
\smallskip

We conclude with a first installment of observations on how to pull information about the components of $\modlad$ from knowledge of the components of $\Rep_\bd (\latrunc)$.
 Suppose that the distinct irreducible components of $\Rep_\bd(\latrunc)$ are 
$$\overline{\Rep_{\latrunc}\SS^{(1)}} = \Filt_{\latrunc}(\SS^{(1)}),\ \ \dots\ ,\  \overline{\Rep_{\latrunc}\SS^{(m)}} = \Filt_{\latrunc}(\SS^{(m)}).$$
Moreover, suppose that $\C$ is an irreducible component of some $\Rep_\la \SS$ with generic module $G$ (recall that, for any $\la$, these components and their generic modules may be algorithmically accessed from quiver and relations of $\la$).  To compare with $\Rep_\bd(\latrunc)$, one first determines which among the $\SS^{(j)}$ govern a filtration of $G$.  Suppose the pertinent sequences are $\SS^{(1)}, \dots, \SS^{(r)}$, that is, $\C \subseteq \Filt_\la \SS^{(j)}$ precisely when $j \le r$.

\begin{observation}  \label{obs6.4}  The closure $\overline{\C}$ is an irreducible component of $\modlad$ if and only if $\,\overline{\C}$ is maximal irreducible in $\Filt_\la \SS^{(j)}$ for all $j \le r$.
\end{observation}

\begin{proof} The claim is immediate from the fact that every irreducible subvariety $\D$ of $\modlad$ which contains $\overline{\C}$ is contained in one of the intersections
$$\begin{matrix}
&&& \qquad \qquad \qquad \qquad \modlad\, \cap\, \Filt_{\latrunc} \SS^{(j)}\  =\  \Filt_\la \SS^{(j)}. &&& \qquad \qquad \qquad\qedhere
\end{matrix}$$
\end{proof}

This leads to a lower bound for the number of irreducible components of $\modlad$.  Computing it in specific instances typically requires a non-negligible effort, as it is not simply based on the {\it number\/} of components of $\Rep_\bd(\latrunc)$.  The bound is sharp in general.  Indeed, if  $\Delta$ denotes the algebra of \ref{ex6.1}(e) and $\la = \Delta / \langle \beta_i \alpha_j \beta_k \mid i,j,k \in \{1,2,3\}\rangle$, then $\Delta = \latrunc$ and the number of irreducible components of $\modlad$ coincides with the lower bound given below.  

\begin{corollary}  \label{cor6.5} Again, let $\bd$ be a dimension vector of a basic $K$-algebra $\la$, and adopt the above notation for the irreducible components of $\Rep_\bd(\latrunc)$. 
Moreover, set
$$A_j : = \Rep_\bd(\latrunc)\  \setminus \bigcup_{i \le m,\ i \ne j} \Filt_{\latrunc} \SS^{(i)} \ \ \ \ \ \text{for} \ j \le m.$$
Then the number of irreducible components of $\modlad$ is bounded from below by the number of $A_j$ which have nonempty intersection with $\modlad$. 
\end{corollary}

\begin{proof}  
Suppose $A_1, \dots, A_s$ are the $A_j$ which intersect $\modlad$ nontrivially, and let $\U_j$ be an irreducible subvariety of $A_j \cap \modlad$ for $j \le s$. Among the $\Filt_{\latrunc} \SS^{(i)}$, the variety $\Filt \SS^{(j)}$ is then the only one to contain $\U_j$.  Consequently, any maximal irreducible subset $\D_j$ of $\modlad$ containing $\U_j$ is an irreducible component of $\modlad$ by the preceding observation. By construction, the resulting $\D_j$ are pairwise different. 
\end{proof}

%%%%%%%%%%%%%%%%%%%%%

%%%%%%%%%%%%%%%%%%%%%%%%%%%%%%%%%%%%%%%%%%%%%%

\end{document}